\documentclass[12pt]{article}%
\usepackage{graphicx}
\usepackage{amsmath}
\usepackage{amsfonts}
\usepackage{amssymb}
\usepackage{algorithm}
\usepackage{algorithmic}
\usepackage{boxedminipage}%
\providecommand{\U}[1]{\protect\rule{.1in}{.1in}}
\newtheorem{theorem}{Theorem}[section]

\newtheorem{definition}[theorem]{Definition}
\newtheorem{example}[theorem]{Example}

\newtheorem{notation}[theorem]{Notation}

\newtheorem{proposition}[theorem]{Proposition}

\newtheorem{remark}[theorem]{Remark}

\newenvironment{proof}[1][Proof]{\noindent \textbf{#1.} }{\  \rule{0.5em}{0.5em}}
\begin{document}

\title{Some Game Theoretic Remarks on \\Two-Player Generalized Cops and Robbers Games}
\author{Athanasios Kehagias and Georgios Konstantinidis}
\date{}
\maketitle

\begin{abstract}
In this paper we study the two-player \emph{generalized Cops and Robber} (GCR)
games introduced by Bonato and MacGillivray. Our main goal is to present a
full, self-contained game theoretic analysis of such games.

\end{abstract}

\section{Introduction\label{sec01}}

In this paper we present a \emph{game theoretic} analysis of the
\emph{two-player Generalized Cops and Robber} (henceforth GCR)
games\ introduced by Bonato and MacGillivray in \cite{Bonato2017}. GCR can be
understood as a general framework for \emph{pursuit games on graphs}. We have
two main goals.

First, while \cite{Bonato2017} presents a very broad and interesting
generalization of the \textquotedblleft classic\textquotedblright\ Cops and
Robbers (henceforth CR) \ game, it follows the tradition of the
CR\ literature, which is dominated by graph theoretic and combinatorial
arguments but pays little attention to core game theoretic concepts such as
the \emph{payoff function}, \emph{value} of a game, \emph{optimal strategies}
etc. Because we consider such concepts essential to the study of games (and in
particular of pursuit games on graphs), in the current paper we will try to
bring them to the foreground.

Secondly, as we will argue in the sequel, eschewing the classic game theoretic
analysis can result in nonrigorous treatment of certain game aspects. Hence
our second goal in using game theoretic tools is to fill certain gaps which
(we believe) exist in the analysis of GCR games \cite{Bonato2017} and even the
classic CR\ game \cite{Hahn2006}.

The \textquotedblleft classic\textquotedblright\ CR\ game was introduced
(independently)\ in \cite{Nowakowski1983} and
\cite{Quilliot1978,Quilliot1983,Quilliot1985} and has been the subject of
intense research ever since. A good and relatively recent review of the
literature appears in the excellent book \cite{Bonato2011} and additional
references can be found in \cite{Fomin2008,Hollinger2011}. As already
mentioned, Bonato and MacGillivray present in \cite{Bonato2017} a very broad
generalization of the classic CR\ game. The abovementioned reviews contain
references to many of the CR\ \textquotedblleft variants\textquotedblright%
\ \ which have appeared in the literature. Almost all such works use a graph
theoretic and/or combinatorial approach; as far as we are aware, very few
authors deal with the game theoretic aspects of both classic CR and its variants.

The rest of this paper is organized as follows. In Section \ref{sec02} we
present definitions and notations which will be used in the sequel. In Section
\ref{sec03} we \emph{solve} the GCR\ game, i.e., we compute its \emph{value}
and \emph{optimal strategies}. We present two different solutions. The
solution of Section \ref{sec0301} is based on a \emph{vertex labeling}
(VL)\ algorithm which has been used in \cite{Bonato2017,Hahn2006}; we prove
\emph{rigorously} that this algorithm computes the value and optimal
strategies (in the game theoretic sense)\ of the GCR\ game; furthermore, this
analysis is \emph{self-contained}, i.e., all our claims are proved within this
paper. The solution of Section \ref{sec0302}, on the other hand, first proves
the existence of value and \emph{afterwards} shows that the value can be
computed by the VL\ algorithm of Section \ref{sec0301}; this analysis is not
self-contained; several classic Game Theoretic results are used, in particular
the Minimax Theorem \cite{Thomas}. In Section \ref{sec04} \ we discuss the
connection of previous analyses of both GCR\ and classic CR to our own. \ We
summarize and conclude in Section \ref{sec05}. Finally, Appendix \ref{secA}
contains some useful, basic facts from Game Theory.

\section{Preliminaries\label{sec02}}

\subsection{Informal Description of the GCR\ Game\label{sec0201}}

We first provide an informal description of the GCR\ game. This description is
inspired by the one given in \cite{Bonato2017} but, as will be explained,
differs in several respects. \ 

To motivate our description, first recall the basic elements of the classic
CR\ game. CR\ is played on a graph $G$, in discrete time steps (\emph{turns}%
)\ and involves two players:\ the \emph{Cop} (or \emph{Pursuer}) and the
\emph{Robber} (or \emph{Evader}). At every turn each player is \emph{located}
on a vertex $v$ of $G$, and a \emph{single} player can move to a new location
(a vertex $v^{\prime}$ adjacent to $v$). The Cop wins if he \textquotedblleft
captures\textquotedblright\ the Robber, i.e., if at some turn they are both
located in the same vertex\footnote{Note that in the usual CR\ description
time is counted in \emph{rounds}, where each round encompasses (in our
terminology) one cop and one robber turn. The two approaches are equivalent.}. 

The \emph{generalized }CR game\emph{ } (GCR\ game) can be informally described
as follows. It involves two \emph{players}, the Pursuer\ and the Evader; the
game is played in turns; in each turn a single player can (subject to certain
restrictions) move between elements of a finite set, resulting in a new
\textquotedblleft configuration\textquotedblright\ of player locations; there
is a set of target (\textquotedblleft capture\textquotedblright%
)\ configurations; the Pursuer wins if a target configuration is reached and
the Evader wins otherwise.

It is easy to see that the classic CR\ is a special case of GCR game. \ As
explained in \cite{Bonato2017} (where details and references are provided) the
definition of GCR\ games also encompasses the following games:\ Distance-$k$
Cops and Robbers, Tandem-win Cops and Robbers, Cops and Robbers with Traps,
Eternal Domination, Revolutionaries and Spies, Seepage and many more.

\subsection{Formal Description of the GCR\ Game\label{sec0202}}

We now give, in several steps, our formal description of the GCR\ game; it is
quite similar but not identical to the one given in \cite{Bonato2017}.

\subsubsection{General Rules\label{sec020201}}

A GCR\ game involves two \emph{players}, $P^{1}$ (Pursuer)\ and $P^{2}$
(Evader); for $n\in\left\{  1,2\right\}  $, $P^{n}$'s possible
\emph{locations} are \ the elements of a finite set $V^{n}$.

A \emph{nonterminal game state} (or \emph{position}) is a triple $\left(
x^{1},x^{2},p\right)  $ where (for $n\in\left\{  1,2\right\}  $)\ $\ x^{n}\in
V^{n}$ indicates $P^{n}$'s \emph{location} and $p$ indicates the \emph{single}
player who currently \textquotedblleft has the move\textquotedblright\ (i.e.
can change his location). The set of all nonterminal game states is
\[
\overline{S}=V^{1}\times V^{2}\times\left\{  1,2\right\}  .
\]
We will also use a \emph{terminal state} $\tau$, so that the full state set is%
\[
S=\overline{S}\cup\left\{  \tau\right\}  .
\]

The game starts at a prespecified \emph{initial state} $s_{0}\in\overline{S}$
and is played in \emph{turns}; in each turn a \emph{single} player moves
(i.e., either changes location or stays in his current location).
\emph{Movement rules} will be described in Section \ref{sec020202}.

A \emph{target set }$S_{c}\subseteq\overline{S}$ is given; $S_{c}$ is the set
of \emph{capture states}. The set of \emph{noncapture states} \ is\emph{
}$S_{nc}=\overline{S}\backslash S_{c}$. So we have the partition%
\[
\overline{S}=S_{nc}\cup S_{c}.
\]
We will also need the sets
\[
\forall n\in\left\{  1,2\right\}  :S^{n}=\left\{  s=\left(  x^{1}%
,x^{2},n\right)  \right\}  ;
\]
i.e., $S^{n}$ is the set of states in which $P^{n}$ has the move. This results
in another partition of $\overline{S}$:%
\[
\overline{S}=S^{1}\cup S^{2}.
\]

It is assumed that both players have \emph{perfect information}, i.e., at
every turn of the game both players have complete knowledge of the way the
game has been played so far.

\begin{example}
\label{prop0201}\normalfont In the CR game $P^{1}$ is the Cop and $P^{2}$ is
the Robber. The location sets are $V^{1}=V^{2}=V$, the vertex set of a graph
$G=\left(  V,E\right)  $. The nonterminal state set is
\[
\overline{S}=\left\{  \left(  x^{1},x^{2},n\right)  :x^{1},x^{2}\in
V,n\in\left\{  1,2\right\}  \right\}  .
\]
The capture set $S_{c}$ is
\[
S_{c}=\left\{  \left(  x^{1},x^{1},n\right)  :x^{1}\in V,n\in\left\{
1,2\right\}  \right\}  .
\]
This is \emph{almost} exactly the classic CR\ game, with one difference:\ we
\ omit the classic \textquotedblleft placement phase\textquotedblright\ in
which first the Cop and then the Robber choose initial positions. We assume
instead that both initial positions (as well as the first player to move)\ are
given by the prespecified initial state $s_{0}=\left(  x_{0}^{1},x_{0}%
^{2},n_{0}\right)  $.
\end{example}

\subsubsection{Movement Rules\label{sec020202}}

At each turn, a player can change his location subject to some movement rules,
which are specific to each particular GCR\ game. For example, in the classic
CR\ game each player can move from his current vertex to any adjacent vertex
(or stay in place).

To specify movement rules, we will often use the following standard game
theoretic notation: for each $n\in\left\{  1,2\right\}  $ we use $-n$ to
indicate the index of the \textquotedblleft other player\textquotedblright%
\ (obviously, $-n$\emph{ }is \emph{not }meant as the negative of $n$). For
example, suppose player $P^{n}$ has location $x^{n}$, then the
\textquotedblleft other\textquotedblright\ player is $P^{-n}$ and has location
$x^{-n}$.

We have already mentioned that each game state specifies which player makes
the next move. Unlike both classic CR\ and the GCR\ of \cite{Bonato2017}%
,\ \emph{we do not demand that the players alternate in taking moves}; in
making a move, a player may take the game to a state in which he again has the move.

For every state $s=\left(  x^{n},x^{-n},n\right)  $ there exists a nonempty
set of possible next moves for $P^{n}$ (the player who has the move). To each
such move corresponds  a unique next state and conversely. Hence we fully
describe possible next moves from state $s$ by the set $N\left(  s\right)  $
of possible next states. For instance, with $s=\left(  x^{n},x^{-n},n\right)
\in\overline{S}$ we have
\[
N\left(  x^{n},x^{-n},n\right)  =\left\{  \left(  y,x^{-n},-n\right)  :y\in
M\left(  x^{n},x^{-n},n\right)  \right\}  .
\]
We assume that the GCR\ game is played for an \emph{infinite} number of
turns\footnote{This is done to conform to a \emph{stochastic game}
formulation, as will be explained in the sequel.}. That is, even after a
capture state is reached, the game will continue ad infinitum, but in a
trivial manner. Namely, in every GCR\ game the only successor of a capture
state is the terminal state:%
\[
s\in S_{c}\Rightarrow N\left(  s\right)  =\left\{  \tau\right\}
\]
which can only transit into itself%
\[
N\left(  \tau\right)  =\left\{  \tau\right\}  .
\]
In short, the movement rules of a particular GCR\ game can be encoded by
\[
\mathbf{N}=\left(  N\left(  s\right)  \right)  _{s\in S},
\]
the collection of possible next moves of each state. Note that $\mathbf{N}$
also specifies (implicitly)\ the state set $S$ which in turn specifies the
location sets $V^{1}$, $V^{2}$.

\begin{example}
\label{prop0202}\normalfont Continuing from Example \ref{prop0201}, in the CR
game played on the graph $G=\left(  V,E\right)  $ we have
\[
\forall s=\left(  x^{n},x^{-n},n\right)  \in\overline{S}:N\left(  s\right)
=\left\{  \left(  z,x^{-p},p\right)  :\left\{  x^{p},z\right\}  \in E\right\}
\text{ and }N\left(  \tau\right)  =\left\{  \tau\right\}  .
\]

\end{example}

\subsubsection{Payoff Functions, Winning Conditions\label{sec020203}}

It remains to define \emph{winning conditions} for the GCR\ game. In
accordance with the classic CR\ game, a reasonable definition is that the
Pursuer wins iff a capture state is reached. To conform with the game
theoretic formulation, this definition should be given, in terms of a
\emph{payoff function}. Actually, we will use a somewhat different definition
which involves \emph{two} payoff functions, as seen below.

\begin{definition}
\label{prop0203}The \emph{turn payoff function} $q\left(  s\right)
\ $specifies the amount gained by the Evader (and lost by the Pursuer)\ at
every turn of GCR\ in which the current state is $s$; it is defined by%
\begin{equation}
q\left(  s\right)  =\left\{
\begin{array}
[c]{ll}%
1 & \text{iff \ }s\in S_{nc}\\
0 & \text{iff \ }s\in S_{c}\cup\left\{  \tau\right\}
\end{array}
\right.  . \label{eq02001}%
\end{equation}

\end{definition}

\begin{definition}
\label{prop0204}The \emph{total payoff function} $Q\left(  s_{0}%
s_{1}...\right)  \ $specifies the amount gained by the Evader (and lost by the
Pursuer)\ in a play of GCR\ with state sequence $s_{0}s_{1}...$; it is defined
by%
\begin{equation}
Q\left(  s_{0}s_{1}s_{2}...\right)  =\sum_{t=0}^{\infty}q\left(  s_{t}\right)
. \label{eq02002}%
\end{equation}

\end{definition}

\noindent The following remarks can be made regarding the significance of the
payoff functions.

\begin{enumerate}
\item When the game \ goes through the state sequence $s_{0}s_{1}...$ ,
clearly $Q\left(  s_{0}s_{1}s_{2}...\right)  $ is the \emph{capture time},
i.e., the number of turns until capture is effected; if $Q\left(  s_{0}%
s_{1}s_{2}...\right)  =\infty$ then capture is never effected.

\item It is also clear that, by the above payoff functions, GCR\ is a
\emph{zero-sum} game. The Pursuer wants to minimize (and the Evader wants to
maximize)\ capture time.

\item Hence we have refined the original winning condition:\ the Pursuer not
only wants to capture the Evader in finite time (and hence win); he wants to
capture in the \emph{shortest possible} time.
\end{enumerate}

\subsubsection{The Stochastic Game GCR\label{sec020204}}

Keeping in mind all of the above, we have the following two definitions.

\begin{definition}
\label{prop0205}A \emph{GCR\ game} is a tuple $\left(  \mathbf{N},S^{1}%
,S^{2},S_{c},s_{0}\right)  $ where

\begin{enumerate}
\item $\mathbf{N}$ describes the movement rules (and implicitly the state set
$S$ and the location sets $V^{1},V^{2}$);

\item $S^{1}$ and $S^{2}$ describe which player has the move in every state;

\item $S_{c}$ describes the capture (winning)\ condition;

\item $s_{0}$ is the initial state.
\end{enumerate}
\end{definition}

\begin{definition}
\label{prop0206}A \emph{GCR\ game family} is a tuple $\left(  \mathbf{N,}%
S^{1},S^{2},S_{c}\right)  $.
\end{definition}

By the above definitions,\ given the game family $\left(  \mathbf{N,}%
S^{1},S^{2},S_{c}\right)  $ and a specific initial state $s_{0}\in S$, we
obtain a particular game $\left(  \mathbf{N},S^{1},S^{2},S_{c},s_{0}\right)
$. In other words, every element of the family $\left(  \mathbf{N,}S^{1}%
,S^{2},S_{c}\right)  $ is a game played by the same rules but from a different
initial state. In Section \ref{sec03} we will mostly prove properties of the
entire family $\left(  \mathbf{N,}S^{1},S^{2},S_{c}\right)  $ which will also
imply corresponding properties of all specific games $\left(  \mathbf{N}%
,S^{1},S^{2},S_{c},s_{0}\right)  $. \ 

Now recall the definition of a stochastic game \cite{filar1996}. It is a
sequence of one-turn games; at every turn a particular one-turn game is
played, the players receive their \emph{turn payoff} and the next game to be
played is selected, depending on the current game and players' actions; the
\emph{total payoff} to each player is the sum of his turn
payoffs\footnote{Despite the term \textquotedblleft
stochastic\textquotedblright, the above definition contains as a special case
a game which evolves in a fully deterministic manner (when all player actions
and transitions to the next game are deterministic).}. Obviously, a
GCR\ \emph{game family }is a \emph{stochastic game} and we could, at this
point, obtain the full GCR\ solution by invoking well known stochastic games
results \cite{filar1996}. However we will avoid this route since, as
mentioned, we want to present a self-contained solution.

At first sight, our definition of GCR\ games and the one given by Bonato and
MacGillivray \cite{Bonato2017} differ in several respects; these differences
will be discussed in Section \ref{sec04}, where we will show that our results
also apply to the Bonato-MacGillivray \ games. Here we only mention what we
consider to be the most important difference. Namely, the GCR\ game of
\cite{Bonato2017} starts with a \textquotedblleft placement
phase\textquotedblright, in which first $P^{1}$ and then $P^{2}$ chooses his
initial location. This yields an initial state $s_{0}$ and the remaining part
of the Bonato-MacGillivray game is basically our $\left(  \mathbf{N}%
,S^{1},S^{2},S_{c},s_{0}\right)  $ game.

\subsection{Additional Game Theoretic Concepts\label{sec0203}}

We conclude this section by presenting some additional standard game theoretic
concepts which will prove useful in the sequel. In what follows we assume that
a game family $\left(  \mathbf{N,}S^{1},S^{2},S_{c}\right)  $ has been
specified, so it is usually omitted from the notation.

\begin{definition}
\label{prop02007}A \emph{history} $h=s_{0}s_{1}...$ is a finite or infinite
sequence of states.
\end{definition}

\begin{definition}
\label{prop02008}We define the following sets of histories%
\begin{align*}
\text{the set of finite length histories}  &  :H_{\ast}=\left\{
h:h=s_{0}s_{1}s_{2}...s_{T}\text{, }T\in\mathbb{N}_{0}\text{, }\forall
t:s_{t}\in S\right\}  ,\\
\text{the set of infinite length histories}  &  :H_{\infty}=\left\{
h:h=s_{0}s_{1}s_{2}...\text{, }\forall t:s_{t}\in S\right\}  .
\end{align*}

\end{definition}

\begin{definition}
\label{prop02009}A (\emph{pure} or \emph{deterministic})\ \emph{strategy
}$\sigma:H_{\ast}\rightarrow S$ is a function which maps finite histories to
\emph{next states}.
\end{definition}

\noindent Regarding the above definition, the following points must be emphasized.

\begin{enumerate}
\item We could also have defined \emph{randomized}\footnote{More precisely,
\emph{behavioral }strategies \cite{filar1996}.} strategies but these will not
be needed in our analysis, because the GCR\ game has perfect information.

\item We deviate from the \textquotedblleft standard\textquotedblright\ pure
strategy definition. In the context of stochastic games, the usual definition
specifies a strategy as a function which maps finite histories to next
\emph{moves}. However, since the GCR\ game evolves deterministically, a move
specifies the next state. Hence our definition is sufficient (and more
convenient) for our purposes.
\end{enumerate}

\begin{definition}
\label{prop02010}A strategy $\sigma^{m}$ is called \emph{positional}\ if it
depends only on the current state $s_{t}$, but neither on previous states nor
on current time $t$, i.e.,
\[
\forall h=s_{0}s_{1}...s_{t},h^{\prime}=s_{0}^{\prime}s_{1}^{\prime
}...,s_{t^{\prime}}^{\prime}:s_{t}=s_{t^{\prime}}^{\prime}=s\Rightarrow
\sigma^{m}\left(  h\right)  =\sigma^{m}\left(  h^{\prime}\right)  =\sigma
^{m}\left(  s\right)  .
\]
A strategy is called \ \emph{nonpositional }iff it is not positional.
\end{definition}

\begin{notation}
\label{prop02011}$H\left(  \sigma^{1},\sigma^{2}|\mathbf{N},S^{1},S^{2}%
,S_{c},s_{0}\right)  $ denotes the infinite history $s_{0}s_{1}s_{2}...$
generated when:\ (i) the game is $\left(  \mathbf{N},S^{1},S^{2},S_{c}%
,s_{0}\right)  $ and (ii)\ for $n\in\left\{  1,2\right\}  $, $P^{n}$ uses
$\sigma^{n}$. When $\left(  \mathbf{N,}S^{1},S^{2},S_{c}\right)  $ is
understood from the context, we simply write $H\left(  \sigma^{1},\sigma
^{2}|s_{0}\right)  .$
\end{notation}

\begin{notation}
\label{prop02012}$T\left(  \sigma^{1},\sigma^{2}|\mathbf{N},S^{1},S^{2}%
,S_{c},s_{0}\right)  $ denotes the \emph{capture time}, i.e., the first
(actually the only)\ time at which a capture state $s\in S_{c}$ is reached
when:\ (i) the game is $\left(  \mathbf{N},S^{1},S^{2},S_{c},s_{0}\right)  $
and (ii)$\ $\ for $n\in\left\{  1,2\right\}  $, $P^{n}$ uses $\sigma^{n}$.
When $\left(  \mathbf{N,}S^{1},S^{2},S_{c}\right)  $ is understood from the
context, we simply write $T\left(  \sigma^{1},\sigma^{2}|s_{0}\right)  $.
\end{notation}

\begin{remark}
\label{prop02013}\normalfont It is obvious that
\[
T\left(  \sigma^{1},\sigma^{2}|\mathbf{N},S^{1},S^{2},S_{c},s_{0}\right)
=Q\left(  H\left(  \sigma^{1},\sigma^{2}|\mathbf{N},S^{1},S^{2},S_{c}%
,s_{0}\right)  \right)  .
\]

\end{remark}

\bigskip

\noindent We conclude by presenting several well-known definitions and facts
about zero-sum games in a notation specific to a GCR\ game\footnote{The
general form of these facts appears in Appendix \ref{secA}.}. What follows
concerns a specific game $\left(  \mathbf{N,}S^{1},S^{2},S_{c}\right)  $,
which we assume known and (for brevity)\ do not include in the notation.

\begin{definition}
\label{prop02014}For every $s$ we define the following two quantities%
\begin{align*}
\text{\emph{lower value} of }\left(  \mathbf{N,}S^{1},S^{2},S_{c}\right)   &
:T^{-}\left(  s\right)  =\sup_{\sigma^{2}}\inf_{\sigma^{1}}T\left(  \sigma
^{1},\sigma^{2}|s\right)  ,\\
\text{\emph{upper value} of }\left(  \mathbf{N,}S^{1},S^{2},S_{c}\right)   &
:T^{+}\left(  s\right)  =\inf_{\sigma^{1}}\sup_{\sigma^{2}}T\left(  \sigma
^{1},\sigma^{2}|s\right)  .
\end{align*}

\end{definition}

\begin{proposition}
\label{prop02015}For every $s$ we have
\[
T^{-}\left(  s\right)  =\sup_{\sigma^{2}}\inf_{\sigma^{1}}T\left(  \sigma
^{1},\sigma^{2}|s\right)  \leq\inf_{\sigma^{1}}\sup_{\sigma^{2}}T\left(
\sigma^{1},\sigma^{2}|s\right)  =T^{+}\left(  s\right)  .
\]

\end{proposition}

\begin{definition}
\label{prop02016}We say that game $\left(  \mathbf{N,}S^{1},S^{2}%
,S_{c}\right)  $ has a \emph{value} $\widehat{T}\left(  s\right)  \ $iff
\[
T^{-}\left(  s\right)  =\sup_{\sigma^{2}}\inf_{\sigma^{1}}T\left(  \sigma
^{1},\sigma^{2}|s\right)  =\inf_{\sigma^{1}}\sup_{\sigma^{2}}T\left(
\sigma^{1},\sigma^{2}|s\right)  =T^{+}\left(  s\right)  .
\]
in which case we define $\widehat{T}\left(  s\right)  $\ to be%
\[
\widehat{T}\left(  s\right)  =\sup_{\sigma^{2}}\inf_{\sigma^{1}}T\left(
\sigma^{1},\sigma^{2}|s\right)  =\inf_{\sigma^{1}}\sup_{\sigma^{2}}T\left(
\sigma^{1},\sigma^{2}|s\right)  .
\]

\end{definition}

\begin{proposition}
\label{prop02016a}We say that the game $\left(  \mathbf{N,}S^{1},S^{2}%
,S_{c},s\right)  $ has \emph{optimal strategies} $\widehat{\sigma}%
^{1},\widehat{\sigma}^{2}$ iff
\[%
\begin{array}
[c]{l}%
\forall\sigma^{2}:\text{ }T\left(  \widehat{\sigma}^{1},\sigma^{2}|s\right)
\leq T\left(  \widehat{\sigma}^{1},\widehat{\sigma}^{2}|s\right)  ,\\
\forall\sigma^{1}:T\left(  \sigma^{1},\widehat{\sigma}^{2}|s\right)
\geq\widehat{T}\left(  \widehat{\sigma}^{1},\widehat{\sigma}^{2}|s\right)  .
\end{array}
\]

\end{proposition}

\begin{proposition}
\label{prop02017}The game $\left(  \mathbf{N,}S^{1},S^{2},S_{c},s\right)  $
has value $\widehat{T}\left(  s\right)  $ and optimal strategies
$\widehat{\sigma}^{1},\widehat{\sigma}^{2}$ iff
\[%
\begin{array}
[c]{l}%
\forall\sigma^{2}:\text{ }T\left(  \widehat{\sigma}^{1},\sigma^{2}|s\right)
\leq\widehat{T}\left(  s\right)  ,\\
\forall\sigma^{1}:T\left(  \sigma^{1},\widehat{\sigma}^{2}|s\right)
\geq\widehat{T}\left(  s\right)  .
\end{array}
\]

\end{proposition}

\section{Solution of the GCR\ Game\label{sec03}}

In Section \ref{sec0202} we have formulated GCR\ as a two-player, zero-sum
stochastic game family $\left(  \mathbf{N,}S^{1},S^{2},S_{c}\right)  $. By
\textquotedblleft solving $\left(  \mathbf{N,}S^{1},S^{2},S_{c}\right)
$\textquotedblright\ we mean proving that, for every $s\in\overline{S}$, the
game $\left(  \mathbf{N,}S^{1},S^{2},S_{c},s\right)  $ has value and optimal
strategies, as well as computing these quantities. All of these things can be
achieved by using standard results regarding stochastic games \cite{filar1996}%
. However, as mentioned, we want to provide a self-contained solution.
Actually, we will provide two such solutions, one in Section \ref{sec0301} and
another in Section \ref{sec0302}.

\subsection{Vertex Labeling Solution\label{sec0301}}

Our first solution of $\left(  \mathbf{N},S^{1},S^{2},S_{c}\right)  $ is
obtained by studying the properties of a \emph{vertex labeling} algorithm
(VL\ algorithm) which is a modification of an algorithm first presented in
\cite{Hahn2006} (for the CR\ game) and then in \cite{Bonato2017} (for the
GCR\ game). Our analysis is more detailed than the ones presented in
\cite{Hahn2006,Bonato2017}, mainly because we spell out all the game theoretic
aspects missing from \cite{Hahn2006,Bonato2017}. The differences between our
approach and that of \cite{Hahn2006,Bonato2017} will be discussed in Section
\ref{sec04}.

\begin{algorithm}[H]
\caption{\textbf{: The Vertex Labeling Algorithm}} \begin{algorithmic}[1]
\REQUIRE The game family $\left(\mathbf{N},S^1,S^2,S_c\right)$.
\FOR{$s\in \overline{S}$}
\IF{$s\in S_{c}$}
\STATE $T^{0}\left(  s\right)  =0$
\ELSE
\STATE $T^{0}\left(  s\right)  =\infty$
\ENDIF
\ENDFOR
\FOR{$i=1,2,..$}
\FOR{$s\in \overline{S}$}
\IF{$T^{i-1}\left(  s\right)  <\infty$}
\STATE $T^{i}\left(  s\right)  =T^{i-1}\left(  s\right)$
\ELSIF{$s\in S^{1}$}
\STATE $T^{i}\left(s\right)  =1+\min_{s^{\prime}\in N\left(  s\right)  }T^{i-1}\left(s^{\prime}\right)  $
\ELSIF{$s\in S^{2}$}
\STATE $T^{i}\left(s\right)  =1+\max_{s^{\prime}\in N\left(  s\right)  }T^{i-1}\left(s^{\prime}\right)  $
\ENDIF
\ENDFOR
\ENDFOR
\end{algorithmic}
\end{algorithm}

\noindent The VL\ algorithm produces, for every $s\in S$, an infinite sequence%
\[
T^{0}(s),T^{1}(s),T^{2}(s),....
\]
The only two forms which $\left(  T^{i}(s)\right)  _{i\in\mathbb{N}_{0}}$ can
take are:\ 

\begin{enumerate}
\item for every $i\in\mathbb{N}_{0}$, $T^{i}(s)$ has the same value (which is
either $0$ or $\infty$),

\item there exists some $n\in%
\mathbb{N}
$ such that $T^{i}(s)=\infty$ for $i\in\{0,...,n-1\}$ and $T^{i}(s)=m\in%
\mathbb{N}
$ for $i\in\{n,n+1,...\}$.
\end{enumerate}

\noindent Hence the following are well defined (with the understanding that
$\min\emptyset=\infty$).

\begin{definition}
\label{prop03001}For all $s\in\overline{S}$ we define
\begin{align}
\overline{T}(s)  &  =\lim_{i\rightarrow\infty}T^{i}(s),\label{eq03000}\\
\widetilde{T}(s)  &  =\min\{i:T^{i}(s)<\infty\}. \label{eq03001}%
\end{align}

\end{definition}

\noindent Our next goal is to show that, for all $s\in\overline{S\noindent}$,
$\overline{T}\left(  s\right)  =\widetilde{T}\left(  s\right)  $. In other
words, if $T^{i}\left(  s\right)  $ attains a finite value at the $i$-th
iteration of the VL\ algorithm, this value (which equals the limit
$\overline{T}\left(  s\right)  $)\ will be $i$. We need the following
auxiliary proposition.

\begin{proposition}
\label{prop03002}For all $s\in\overline{S\noindent}$ and for all
$n\in\mathbb{N}_{0}$ we have
\begin{equation}
\left(  \widetilde{T}\left(  s\right)  =n\right)  \Rightarrow\left(  \forall
i\geq n:T^{i}\left(  s\right)  =\overline{T}\left(  s\right)  \leq n\right)  .
\label{eq03036}%
\end{equation}

\end{proposition}

\begin{proof}
Let us first prove that, for all $s\in\overline{S\noindent}$ and
$n\in\mathbb{N}_{0}$, we have
\begin{equation}
\left(  \widetilde{T}\left(  s\right)  =n\right)  \Rightarrow T^{n}\left(
s\right)  \leq n. \label{eq03033}%
\end{equation}
Clearly (\ref{eq03033}) holds for $n=0$; assume it holds for $n\in\left\{
0,1,...,k\right\}  $. Now, if for some $s$ we have $\widetilde{T}\left(
s\right)  =k+1$ then $T^{k}\left(  s\right)  =\infty$ and $T^{k+1}\left(
s\right)  <\infty$. From this and lines 10-16 of the VL\ algorithm, we see
that
\[
T^{k+1}\left(  s\right)  =1+T^{k}\left(  s^{\prime}\right)
\]
for some $s^{\prime}\in N\left(  s\right)  $ which satisfies $T^{k}\left(
s^{\prime}\right)  =m<\infty$. But then $\widetilde{T}\left(  s^{\prime
}\right)  \leq k$\ which, by the inductive assumption, implies $T^{k}\left(
s^{\prime}\right)  \leq k$. Hence $T^{k+1}\left(  s\right)  \leq k+1$ and we
have proved (\ref{eq03033}) for every $n\in\mathbb{N}_{0}$. Now (\ref{eq03036}%
) follows from lines 10-11 of the VL\ algorithm.
\end{proof}

\begin{proposition}
\label{prop03003}For every $s\in\overline{S}$ we have:%
\begin{equation}
\overline{T}\left(  s\right)  =\widetilde{T}\left(  s\right)  .
\label{eq03005}%
\end{equation}

\end{proposition}

\begin{proof}
We partition $\overline{S}$ into the following two sets%
\[
S_{a}=\left\{  s:\widetilde{T}\left(  s\right)  =\infty\right\}  ,\qquad
S_{b}=\left\{  s:\widetilde{T}\left(  s\right)  <\infty\right\}  .
\]
If $s\in S_{a}$, then $\widetilde{T}(s)=\infty$; then%
\[
\{i:T^{i}\left(  s\right)  <\infty\}=\emptyset\Rightarrow\left(  \forall
i:T^{i}\left(  s\right)  =\infty\right)  ;
\]
hence $\overline{T}(s)=\infty$. We conclude that
\begin{equation}
\forall s\in S_{a}:\text{ }\overline{T}\left(  s\right)  =\widetilde{T}\left(
s\right)  . \label{eq03031}%
\end{equation}
To complete the proof of (\ref{eq03005}), we must prove
\begin{equation}
\forall s\in S_{b}:\text{ }\overline{T}\left(  s\right)  =\widetilde{T}\left(
s\right)  . \label{eq03032}%
\end{equation}
To this end we will show that, for all $n\in\mathbb{N}_{0}$, we have%
\begin{equation}
\forall s\in S_{b}:\text{ }\widetilde{T}\left(  s\right)  =n\Rightarrow
\overline{T}\left(  s\right)  =T^{n}\left(  s\right)  =T^{n+1}\left(
s\right)  =...=n. \label{eq03034}%
\end{equation}
Now, (\ref{eq03034}) clearly holds for $n=0$. Suppose it holds for
$n\in\left\{  0,1,...,k\right\}  $. Take any state $s\in S_{b}$ such that
$\widetilde{T}\left(  s\right)  =k+1$. Then we have%
\begin{align}
T^{k}\left(  s\right)   &  =\infty,\text{ }\label{eq03035}\\
T^{k+1}\left(  s\right)   &  =T^{k+2}\left(  s\right)  =...=\overline
{T}\left(  s\right)  =m\leq k+1<\infty. \label{eq03035a}%
\end{align}

\begin{enumerate}
\item If $s\in S^{1}$, (\ref{eq03035})-(\ref{eq03035a}) imply that there
exists some $s^{\prime}\in N\left(  s\right)  $ such that%
\[
T^{k}\left(  s^{\prime}\right)  =\min_{u\in N\left(  s\right)  }T^{k}\left(
u\right)  \qquad\text{and\qquad}T^{k+1}\left(  s\right)  =1+T^{k}\left(
s^{\prime}\right)  .
\]
Also (\ref{eq03035})$\ $implies that $\min_{u\in N\left(  s\right)  }%
T^{k-1}\left(  u\right)  =\infty$ and hence
\begin{equation}
T^{k-1}\left(  s^{\prime}\right)  =\infty. \label{eq03035b}%
\end{equation}
But $T^{k}\left(  s^{\prime}\right)  =T^{k+1}\left(  s\right)  -1<\infty$;
hence $\widetilde{T}\left(  s^{\prime}\right)  =k$ and, using the inductive
assumption, we have
\begin{align*}
\overline{T}\left(  s^{\prime}\right)   &  =T^{k}\left(  s^{\prime}\right)
=T^{k+1}\left(  s^{\prime}\right)  =...=k\Rightarrow\\
\overline{T}\left(  s\right)   &  =T^{k+1}\left(  s\right)  =T^{k+2}\left(
s\right)  =...=1+T^{k}\left(  s^{\prime}\right)  =1+k.
\end{align*}

\item If $s\in S^{2}$, (\ref{eq03035a}) implies that $\max_{u\in N\left(
s\right)  }T^{k}\left(  u\right)  <\infty$; hence for all $u\in N\left(
s\right)  $ we have $T^{k}\left(  u\right)  <\infty$. On the other hand,
(\ref{eq03035}) implies that there exists some $s^{\prime}\in N\left(
s\right)  $ such that $T^{k-1}\left(  s^{\prime}\right)  =\infty$. Hence
$\widetilde{T}\left(  s^{\prime}\right)  =k$. We also have
\begin{align*}
T^{k+1}\left(  s\right)  =m<\infty &  \Rightarrow\left(  \forall u\in N\left(
s\right)  :T^{k}\left(  u\right)  <\infty\right) \\
&  \Rightarrow\left(  \forall u\in N\left(  s\right)  :\widetilde{T}\left(
u\right)  \leq k\right) \\
&  \Rightarrow\left(  \forall u\in N\left(  s\right)  :\overline{T}\left(
u\right)  =T^{k}\left(  u\right)  \leq k\right)
\end{align*}
(the last line following from Proposition \ref{prop03002}). Hence $s^{\prime}$
achieves the maximum:
\[
T^{k}\left(  s^{\prime}\right)  =\max_{u\in N\left(  s\right)  }T^{k}\left(
u\right)  =k
\]
and so
\[
\overline{T}\left(  s\right)  =T^{k+1}\left(  s\right)  =T^{k+2}\left(
s\right)  =...=1+T^{k}\left(  s^{\prime}\right)  =1+k.
\]

\end{enumerate}

\noindent Thus we have proved that (\ref{eq03034}) holds for every
$n\in\mathbb{N}_{0}$, and we are done.
\end{proof}

\bigskip

\bigskip

\noindent We also have the following useful and easily provable proposition.

\begin{proposition}
\label{prop03004}The collection $\left(  \overline{T}\left(  s\right)
\right)  _{s\in\overline{S}}$ satisfies the following:%
\begin{align}
\forall s  &  \in S^{1}\backslash S_{c}:\overline{T}\left(  s\right)
=1+\min_{s^{\prime}\in\mathbf{N}(s)}\overline{T}\left(  s^{\prime}\right)
,\label{eq03002}\\
\forall s  &  \in S^{2}\backslash S_{c}:\overline{T}\left(  s\right)
=1+\max_{s^{\prime}\in\mathbf{N}(s)}\overline{T}\left(  s^{\prime}\right)  ,
\label{eq03003}%
\end{align}

\end{proposition}

\begin{proof}
Simply take the limits in lines 13 and 15 of the VL\ algorithm.
\end{proof}

\bigskip

\noindent The system of equations (\ref{eq03002})-(\ref{eq03003}) is the
perfect information version of the \emph{optimality equations} which play a
central role in stochastic games \cite{filar1996}.

\bigskip

\bigskip

\noindent We will next show that, for every $s\in\overline{S}$, \ $\overline
{T}\left(  s\right)  $ is the \emph{value} of the game $\left(  \mathbf{N}%
,S^{1},S^{2},S_{c},s\right)  $. Furthermore, we will use the collection
$\left(  \overline{T}\left(  s\right)  \right)  _{s\in\overline{S}}$ to define
\emph{optimal} strategies for $P^{1}$ and $P^{2}$.

To avoid ambiguities in the definition of the optimal strategies we modify the
functions $\arg\min$ and $\arg\max$ as follows. Given the collection of
successor states $\mathbf{N}=\left(  N\left(  s\right)  \right)
_{s\in\overline{S}}$, we assume that $\overline{S}$ is equipped with a fixed
total order. Then, for every $s\in\overline{S}$, expressions such as
\textquotedblleft the \emph{first} element of $N(s)$ such that
...\textquotedblright\ are well (i.e., uniquely)  defined. Now, for every
function $f:\overline{S}\rightarrow\mathbb{R}\cup\left\{  \infty\right\}  $,
we (re)define $\arg\min$ and $\arg\max$:
\begin{align*}
\forall s &  \in\overline{S}:\arg\min_{s^{\prime}\in N(s)}f\left(  s^{\prime
}\right)  =\text{\textquotedblleft the \emph{first} element }u\in N(s)\text{
s.t. }f\left(  u\right)  =\min_{s^{\prime}\in N(s)}f\left(  s^{\prime}\right)
\text{\textquotedblright,}\\
\forall s &  \in\overline{S}:\arg\max_{s^{\prime}\in N(s)}f\left(  s^{\prime
}\right)  =\text{\textquotedblleft the \emph{first} element }u\in N(s)\text{
s.t. }f\left(  u\right)  =\max_{s^{\prime}\in N(s)}f\left(  s^{\prime}\right)
\text{\textquotedblright.}%
\end{align*}

\begin{definition}
\label{prop03005}Given the game family $\left(  \mathbf{N},S^{1},S^{2}%
,S_{c}\right)  $, with labels $\left(  \overline{T}\left(  s\right)  \right)
_{s\in S_{c}}$, we define the following two strategies%
\begin{align*}
\text{A Pursuer strategy }\widehat{\sigma}^{1} &  :\forall s\in S^{1}%
\backslash S_{c}:\widehat{\sigma}^{1}\left(  s\right)  =\arg\min{}_{s^{\prime
}\in N(s)}\overline{T}\left(  s^{\prime}\right)  ,\\
\text{An Evader strategy }\widehat{\sigma}^{2} &  :\forall s\in S^{2}%
\backslash S_{c}:\widehat{\sigma}^{2}\left(  s\right)  =\arg\max{}_{s^{\prime
}\in N(s)}\overline{T}\left(  s^{\prime}\right)  .
\end{align*}

\end{definition}

\begin{proposition}
\label{prop03006}Given the game family $\left(  \mathbf{N},S^{1},S^{2}%
,S_{c}\right)  $, for every $s\in\overline{S}$: $\overline{T}(s)$ is the value
of the game $\left(  \mathbf{N},S^{1},S^{2},S_{c},s\right)  $, and
$\widehat{\sigma}^{1},\widehat{\sigma}^{2}$ are optimal positional strategies.
\end{proposition}

\begin{proof}
Assuming $\left(  \mathbf{N},S^{1},S^{2},S_{c}\right)  $ given and fixed, we
drop it from all subsequent notation.

It follows immediately from Definition \ref{prop03005} that $\widehat{\sigma
}^{1},\widehat{\sigma}^{2}$ are positional strategies. Given Proposition
\ref{prop02017}, to prove the rest of the theorem it suffices to show that
\begin{align}
\forall s  &  \in S:\forall\sigma^{2}:\text{ }T\left(  \widehat{\sigma}%
^{1},\sigma^{2}|s\right)  \leq\overline{T}(s),\label{eqG001}\\
\forall s  &  \in S:\forall\sigma^{1}:\text{ }T\left(  \sigma^{1}%
,\widehat{\sigma}^{2}|s\right)  \geq\overline{T}(s). \label{eqG002}%
\end{align}
The rest of the proof is divided into two parts.

\medskip

\medskip

\noindent\underline{\textbf{Part I}}. Take any $s$ such that $\overline
{T}(s)<\infty$. \ 

\begin{enumerate}
\item First we show that $\widehat{\sigma}^{1}$ satisfies (\ref{eqG001}), by
showing that, for every $n\in\mathbb{N}_{0}$ we have:%
\begin{equation}
\forall s:\overline{T}(s)=n\Rightarrow\left(  \forall\sigma^{2}:T\left(
\widehat{\sigma}^{1},\sigma^{2}|s\right)  \leq\overline{T}(s)\right)  .
\label{eq03019}%
\end{equation}
Obviously (\ref{eq03019}) holds when $\overline{T}(s)=0$ (because then $s\in
S_{c}$ and $T\left(  \sigma^{1},\sigma^{2}|s\right)  =0$ for all $\sigma
^{1},\sigma^{2}$). Assume it holds for all $n\in\{0,1,...,k\}$ and pick any
$s$ such that $\overline{T}(s)=k+1$.

\begin{enumerate}
\item Suppose $s\in S^{1}$. \ From (\ref{eq03002}) we have%
\[
k+1=1+\min_{s^{\prime}\in N\left(  s\right)  }\overline{T}(s^{\prime
})\Rightarrow\min_{s^{\prime}\in N\left(  s\right)  }\overline{T}(s^{\prime
})=k\text{.}%
\]
Hence, for $\overline{s}=\widehat{\sigma}^{1}\left(  s\right)  =\arg
\min_{s^{\prime}\in N(s)}\overline{T}\left(  s^{\prime}\right)  $ we have
$\ \overline{T}\left(  \overline{s}\right)  =k$. By the inductive assumption,
we then have%
\[
\forall\sigma^{2}:\text{ }T\left(  \widehat{\sigma}^{1},\sigma^{2}%
|\overline{s}\right)  \leq\overline{T}(\overline{s})=k\text{.}%
\]
In other words, if $P^{1}$ uses $\widehat{\sigma}^{1}$ in the game $\left(
\mathbf{N},S^{1},S^{2},S_{c},\overline{s}\right)  $ then he will achieve
capture in at most $k$ moves, no matter how $P^{2}$ plays. Since $\overline
{s}\in N(s)$, $P^{1}$ can also use $\widehat{\sigma}^{1}$ in the game $\left(
\mathbf{N},S^{1},S^{2},S_{c},s\right)  $. This means he will move from $s$ to
$\overline{s}$ and then will play exactly as in game $\left(  \mathbf{N}%
,S^{1},S^{2},S_{c},\overline{s}\right)  $. This ensures capture in at most
$k+1$ moves, no matter how $P^{2}$ plays, i.e.,%
\[
\forall\sigma^{2}:\text{ }T\left(  \widehat{\sigma}^{1},\sigma^{2}|s\right)
\leq k+1\text{.}%
\]
Therefore $\widehat{\sigma}^{1}$ satisfies (\ref{eqG001}) for all $s\in S^{1}$
with $\overline{T}(s)=k+1$.

\item Suppose $s\in S^{2}$. From relation (\ref{eq03003}) we have%
\[
k+1=1+\max_{s^{\prime}\in N\left(  s\right)  }\overline{T}(s^{\prime
})\Rightarrow\max_{s^{\prime}\in N\left(  s\right)  )}\overline{T}(s^{\prime
})=k\text{.}%
\]
Thus, whatever the initial move by player $P^{2}$ in game $\left(
\mathbf{N},S^{1},S^{2},S_{c},s\right)  $, the resulting state $\overline{s}%
\ $will satisfy $\overline{T}\left(  \overline{s}\right)  \leq k$. Again, by
the inductive assumption, we have%
\[
\forall\sigma^{2}:\text{ }T\left(  \widehat{\sigma}^{1},\sigma^{2}%
|\overline{s}\right)  \leq\overline{T}(\overline{s})\leq k
\]
and hence, in the game $\left(  \mathbf{N},S^{1},S^{2},S_{c},\overline
{s}\right)  $, $P^{1}$ can use $\widehat{\sigma}^{1}$ and capture in at most
$k$ moves (no matter what $P^{2}$ plays). We then get, by the same reasoning
as above, that%
\[
\forall\sigma^{2}:\text{ }T\left(  \widehat{\sigma}^{1},\sigma^{2}|s\right)
\leq k+1\text{.}%
\]

Therefore $\widehat{\sigma}^{1}$ satisfies (\ref{eqG001}) for all $s\in S^{2}$
with $\overline{T}(s)=k+1$.
\end{enumerate}

Hence we have shown that $\widehat{\sigma}^{1}$ satisfies (\ref{eqG001}),

\item Next we show inductively that $\widehat{\sigma}^{2}$ satisfies
(\ref{eqG002}). I.e., we show that: for every $n\in\mathbb{N}_{0}$ we have:%
\begin{equation}
\forall s:\overline{T}(s)=n\Rightarrow\left(  \forall\sigma^{1}:T\left(
\sigma^{1},\widehat{\sigma}^{2}|s\right)  \geq\overline{T}(s)\right)  .
\label{eq03020}%
\end{equation}
Obviously (\ref{eq03020})\ holds when $\overline{T}(s)=0$. Assume it holds for
all $n\in\{0,1,...,k\}$ and pick any $s$ such that $\overline{T}(s)=k+1$.

\begin{enumerate}
\item Suppose $s\in S^{2}$. Similarly to the previous case, from
(\ref{eq03003}) we get $\max_{s^{\prime}\in N\left(  s\right)  }\overline
{T}(s^{\prime})=k$. Hence for $\overline{s}=\widehat{\sigma}^{2}\left(
s\right)  =\arg\max_{s^{\prime}\in N\left(  s\right)  }\overline{T}(s^{\prime
})$ we have $\overline{T}(\overline{s})=k$. By the inductive assumption we
then have
\[
\forall\sigma^{1}:T\left(  \sigma^{1},\widehat{\sigma}^{2}|\overline
{s}\right)  \geq\overline{T}(\overline{s})=k\text{.}%
\]
Hence $P^{2}$ using $\widehat{\sigma}^{2}$ in the game $\left(  \mathbf{N}%
,S^{1},S^{2},S_{c},\overline{s}\right)  $ can ensure capture will take $k$
turns or more, no matter how $P^{1}$ plays. Reasoning as previously, $P^{2}$
using $\widehat{\sigma}^{2}$ in $\left(  \mathbf{N},S^{1},S^{2},S_{c}%
,s\right)  $ can ensure capture will take $k+1$ turns or more, no matter how
$P^{1}$ plays, i.e.,%
\[
\forall\sigma^{1}:T\left(  \sigma^{1},\widehat{\sigma}^{2}|s\right)  \geq
k+1.
\]
Therefore $\widehat{\sigma}^{2}$ satisfies (\ref{eqG002}) for all $s\in S^{2}$
with $\overline{T}(s)=k+1$.

\item Finally, suppose $s\in S^{1}$; by an argument similar to that of case
1.b above we get that%
\[
\forall\sigma^{1}:T\left(  \sigma^{1},\widehat{\sigma}^{2}|s\right)  \geq
k+1.
\]
Therefore $\widehat{\sigma}^{2}$ satisfies (\ref{eqG002}) for all $s\in S^{1}$
with $\overline{T}(s)=k+1$.
\end{enumerate}

Hence we have shown that $\widehat{\sigma}^{2}$ satisfies (\ref{eqG002}).
\end{enumerate}

We have completed the proof of (\ref{eqG001})-(\ref{eqG002}) for all $s$ such
that $\overline{T}(s)<\infty$.

\medskip

\medskip

\noindent\underline{\textbf{Part II}}. Now take any state $s$ such that
$\overline{T}(s)=\infty$. Then we clearly have
\[
\forall s\in S:\forall\sigma^{2}:\text{ }T\left(  \widehat{\sigma}^{1}%
,\sigma^{2}|s\right)  \leq\overline{T}(s)=\infty;
\]
which proves (\ref{eqG001}). \ To prove (\ref{eqG002}), first set $s_{0}=s$
and then, for any Pursuer strategy $\sigma^{1}$, let%
\[
s_{0}s_{1}s_{2}...=H\left(  \sigma^{1},\widehat{\sigma}^{2}|s_{0}\right)  .
\]
We want to show that $s_{0}s_{1}s_{2}...$ never reaches a capture state. We
will actually prove something (apparently) stronger:%
\begin{equation}
\forall t\in\mathbb{N}_{0},s_{t}\in H(\sigma^{1},\widehat{\sigma}^{2}%
|s_{0}):\overline{T}\left(  s_{t}\right)  =\infty\text{.} \label{eq03037}%
\end{equation}
Suppose (\ref{eq03037}) is false and let $s_{k+1}$ be the first element of
$s_{0}s_{1}s_{2}...$ with $\overline{T}\left(  s_{t}\right)  <\infty$. I.e.,
\begin{equation}
\forall t\leq k:\overline{T}\left(  s_{t}\right)  =\infty\text{ and }%
\overline{T}\left(  s_{k+1}\right)  =m<\infty\text{.} \label{eq03038}%
\end{equation}
If $s_{k}\in S^{2}$, then $s_{k+1}=\widehat{\sigma}^{2}\left(  s_{k}\right)
=\arg\max_{s^{\prime}\in N\left(  s_{k}\right)  }\overline{T}\left(
s^{\prime}\right)  $ and
\begin{equation}
\infty=\overline{T}\left(  s_{k}\right)  =1+\overline{T}\left(  s_{k+1}%
\right)  \Rightarrow\overline{T}\left(  s_{k+1}\right)  =\infty.
\label{eq03039}%
\end{equation}
If $s_{k}\in S^{1}$, then clearly (for any $\sigma^{1}$)$\ s_{k+1}\in N\left(
s_{k}\right)  $. Since $\overline{T}\left(  s_{k+1}\right)  =m<\infty$ we will
have
\begin{equation}
\overline{T}\left(  s_{k}\right)  =1+\min_{s^{\prime}\in N\left(
s_{k}\right)  }\overline{T}\left(  s^{\prime}\right)  \leq1+\overline
{T}\left(  s_{k+1}\right)  =1+m<\infty. \label{eq03040}%
\end{equation}
In both (\ref{eq03039}) and (\ref{eq03040})\ we have a contradiction. \ We
conclude that we cannot have $\overline{T}\left(  s_{t}\right)  <\infty$ for
any $t$; consequently (\ref{eq03037}) is true. Hence we have completed the
proof of (\ref{eqG001})-(\ref{eqG002}) for all $s$ such that $\overline
{T}(s)=\infty$.
\end{proof}

\begin{remark}
\normalfont It is worth noting that the VL\ algorithm implements both
\emph{Backward Induction} and \emph{Value Iteration}.
\end{remark}

\subsection{An Alternative Solution\label{sec0302}}

We will now present an alternative approach to the solution of the
GCR$\ $game. While this approach reaches the same conclusions as the VL-based
approach of Section \ref{sec0301}, it underscores certain aspects of the
GCR$\ $game which are not obvious in the VL-based approach. On the other hand,
the new approach is not totally self-contained; in particular it invokes Von
Neumann's famous MinMax Theorem \cite{Thomas}.

There is another sense in which the approach of this section \emph{appears}
weaker than the VL-based one. Namely, \emph{we will henceforth assume here
that the two players move alternately}, i.e., that every move by $P^{n}$ is
followed a move by $P^{-n}$. However this assumption is introduced only for
clarity of presentation; the proofs can be modified so that they hold when the
alternate-moves assumption is removed.

We start by introducing the concepts of \emph{Pursuer-win} and
\emph{Evader-win} games, which are generalizations of the well known concepts
of cop-win and robber-win graphs.

\begin{definition}
\label{prop03007}The game $\left(  \mathbf{N},S^{1},S^{2},S_{c},s\right)  $ is
called \emph{Pursuer-win} (\emph{P-win}) iff $P^{1}$ has \ a strategy
$\overline{\sigma}^{1}$ which effects capture for every Evader strategy
$\sigma^{2}$; $\left(  \mathbf{N},S^{1},S^{2},S_{c},s\right)  $ is called
\emph{Evader-win} (\emph{E-win}) iff $P^{2}$ has \ a strategy $\overline
{\sigma}^{2}$ which avoids capture for every Pursuer strategy $\sigma^{1}$. In
other words%
\begin{align}
\left(  \mathbf{N},S^{1},S^{2},S_{c},s\right)  \text{ is \emph{P-win} iff}  &
:\exists\overline{\sigma}^{1}:\forall\sigma^{2}:T\left(  \overline{\sigma}%
^{1},\sigma^{2}|s\right)  <\infty,\label{eq05001}\\
\left(  \mathbf{N},S^{1},S^{2},S_{c},s\right)  \text{ is \emph{E-win} iff}  &
:\exists\overline{\sigma}^{2}:\forall\sigma^{1}:T\left(  \sigma^{1}%
,\overline{\sigma}^{2}|s\right)  =\infty\label{eq05002}%
\end{align}

\end{definition}

\noindent A point that is not often stressed in the discussion of cop-win and
robber-win graphs is that we cannot \emph{automatically} conclude that a graph
is either cop-win or robber-win. More generally, we cannot automatically
conclude that the game $\left(  \mathbf{N},S^{1},S^{2},S_{c},s\right)  $ is
either P-win or E-win, because the opposite of (\ref{eq05001}) is%
\begin{equation}
\forall\sigma^{1}:\exists\sigma_{\sigma^{1}}^{2}:T\left(  \sigma^{1}%
,\sigma_{\sigma^{1}}^{2}|s\right)  =\infty\label{eq05003}%
\end{equation}
(i.e., the Evader strategy $\sigma_{\sigma^{1}}^{2}$ which ensures no capture
takes place will in general depend on the Pursuer strategy $\sigma^{1}$); and
(\ref{eq05003}) is not equivalent to (\ref{eq05002}). However, we can
\emph{prove} that, indeed, any $\left(  \mathbf{N},S^{1},S^{2},S_{c},s\right)
$ is either E-win or P-win.

\begin{proposition}
\label{prop03008}Every $\left(  \mathbf{N},S^{1},S^{2},S_{c},s\right)  $ is
either P-win or E-win.
\end{proposition}

\begin{proof}
Pick any $\left(  \mathbf{N},S^{1},S^{2},S_{c},s\right)  $ which is not P-win.
We will show that it is E-win by constructing an Evader strategy
$\overline{\sigma}^{2}$ which satisfies (\ref{eq05002}).

Suppose that $s$ belongs to $S^{1}$ and let $s_{0}=s$. Further, suppose that
$P^{1}$ uses some strategy by which the game moves from $s_{0}$ to some
$s_{1}$. From (\ref{eq05003}) we know that for \emph{any} such $s_{1}$,
$P^{2}$ has at least one strategy by which he can ensure capture never takes
place. Suppose by some such strategy the game moves to some $s_{2}$ (which
will in general depend on $s_{0}$ and $s_{1}$). For all possible $s_{1}$'s
define
\[
\overline{\sigma}^{2}\left(  s_{0}s_{1}\right)  =s_{2}.
\]
This defines the part of $\overline{\sigma}^{2}$ which applies to histories of
length two, i.e., the ones belonging to
\[
\left\{  h=s_{0}s_{1}:s_{1}\in N\left(  s_{0}\right)  \right\}  .
\]
Now, from $s_{2}=\overline{\sigma}^{2}\left(  s_{0}s_{1}\right)  $ the game
can (depending on $P^{1}$'s strategy) move to any $s_{3}\in N\left(
s_{2}\right)  $; since $\left(  \mathbf{N},S^{1},S^{2},S_{c},s_{0}\right)  $
is not P-win, $P^{2}$ has at least one strategy by which he can ensure capture
never takes place from $s_{3}$. Suppose by some such strategy the game moves
to some $s_{4}$ (which will in general depend on $s_{0}s_{1}s_{2}s_{3}$). For
all possible $s_{3}$'s define
\[
\overline{\sigma}^{2}\left(  s_{0}s_{1}s_{2}s_{3}\right)  =s_{4}.
\]
This defines the part of $\overline{\sigma}^{2}$ which applies to all
histories belonging to
\[
\left\{  h=s_{0}s_{1}s_{2}s_{3}:s_{t}\in N\left(  s_{t-1}\right)  \text{ for
}t\in\left\{  1,2,3\right\}  \right\}  .
\]
Continuing in this manner we can extend the definition of $\overline{\sigma
}^{2}$ to all legal histories starting at $s_{0}$ and having finite length;
where at each step $P^{2}$ chooses a move not leading to a P-win state. In
other words, the extension is such that: for every finite history $s_{0}%
s_{1}...s_{2n+1}$, $\overline{\sigma}\left(  s_{0}s_{1}...s_{2n+1}\right)  $
leads to a state from which $P^{2}$ can ensure capture never takes place. We
can complete the definition of $\overline{\sigma}^{2}$ (on the rest of its
domain $H_{\ast}$) by specifying legal but otherwise arbitrary moves. Hence we
have constructed a $\overline{\sigma}^{2}$ which ensures that capture will
never take place, no matter how $P^{1}$ plays. In other words we have shown
that%
\[
\exists\overline{\sigma}^{2}:\forall\sigma^{1}:T\left(  \sigma_{1}%
,\overline{\sigma}^{2}|s_{0}\right)  =\infty
\]
which is exactly (\ref{eq05002}) and shows that $\left(  \mathbf{N}%
,S^{1},S^{2},S_{c},s_{0}\right)  $ is E-win. So we have shown that:\ for every
$s_{0}\in S^{1}$, if $\left(  \mathbf{N},S^{1},S^{2},S_{c},s_{0}\right)  $ is
not P-win, then it is E-win.

By a similar argument we reach the same conclusion for all $s_{0}\in S^{2}$;
hence the proof is complete.
\end{proof}

The next proposition shows that:\ (i)\ if $\left(  \mathbf{N},S^{1}%
,S^{2},S_{c},s\right)  $ is P-win then there is a Pursuer strategy which
provides an upper bound (\emph{valid for all }$\sigma^{2}$\emph{ strategies})
on the capture time; (ii)\ if $\left(  \mathbf{N},S^{1},S^{2},S_{c},s\right)
$ is E-win then there is an Evader strategy which ensures capture never takes
place. Our proof is an adaptation of Zermelo's proof of a similar proposition
regarding \emph{chess} \cite{Konig1927,Schwalbe2001}. While our proof is more
detailed than Zermelo's, it is based on his basic ideas.

\begin{proposition}
\label{prop03009}If $\left(  \mathbf{N},S^{1},S^{2},S_{c},s\right)  $ is
P-win, then
\begin{equation}
\exists T_{a}\left(  s\right)  :\exists\overline{\sigma}^{1}:\sup_{\sigma^{2}%
}T\left(  \overline{\sigma}^{1},\sigma^{2}|s\right)  \leq T_{a}\left(
s\right)  <\infty. \label{eq05004}%
\end{equation}
If $\left(  \mathbf{N},S^{1},S^{2},S_{c},s\right)  $ is E-win, then%
\begin{equation}
\exists\overline{\sigma}^{2}:\inf_{\sigma^{1}}T\left(  \sigma^{1}%
,\overline{\sigma}^{2}|s\right)  =\infty\label{eq05005}%
\end{equation}

\end{proposition}

\begin{proof}
Let us define the following sets%
\begin{align*}
S_{p}  &  =\left\{  s:\exists\overline{\sigma}^{1}:\text{(i)\ }\forall
\sigma^{2}:T\left(  \overline{\sigma}^{1},\sigma^{2}|s\right)  <\infty\text{
and (ii)\ }\sup_{\sigma^{2}}T\left(  \overline{\sigma}^{1},\sigma
^{2}|s\right)  <\infty\right\}  ,\\
S_{q}  &  =\left\{  s:\exists\overline{\sigma}^{1}:\text{(i)\ }\forall
\sigma^{2}:T\left(  \overline{\sigma}^{1},\sigma^{2}|s\right)  <\infty\text{
and (ii)\ }\sup_{\sigma^{2}}T\left(  \overline{\sigma}^{1},\sigma
^{2}|s\right)  =\infty\right\}  ,\\
S_{r}  &  =\left\{  s:\exists\overline{\sigma}^{2}:\forall\sigma^{1}:T\left(
\sigma^{1},\overline{\sigma}^{2}|s\right)  =\infty\right\}  .
\end{align*}
Clearly $S_{c}\subseteq S_{p}$. Also $S_{p}\cup S_{q}$ (resp. $S_{r}$)\ is the
set of all states such that $\left(  \mathbf{N},S^{1},S^{2},S_{c},s\right)  $
is P-win (resp. E-win). It follows from Proposition \ref{prop03008}, that the
complement of $S_{r}$ is $S_{p}\cup S_{q}$. Also, clearly, $S_{p}\cap
S_{q}=\emptyset$. Hence $S_{p},S_{q},S_{r}$ form a partition of $\overline{S}$.

We will prove by contradiction that $S_{q}=\emptyset$ and hence conclude that
(\ref{eq05004}) is true. \ So suppose that there exists $s_{0}\in S_{q}\cap
S^{1}$. If the Pursuer has a move from $s_{0}$ into some $s_{1}\in S_{p}$,
then
\[
\left(  \exists\overline{\sigma}^{1}:\sup_{\sigma^{2}}T\left(  \overline
{\sigma}^{1},\sigma^{2}|s_{1}\right)  <\infty\right)  \Rightarrow\left(
\exists\widetilde{\sigma}^{1}:\sup_{\sigma^{2}}T\left(  \widetilde{\sigma}%
^{1},\sigma^{2}|s_{0}\right)  <\infty\right)  \Rightarrow s\not \in S_{q}%
\]
which contradicts the initial assumption. Hence all the Pursuer's moves must
lead to some $s_{1}\in\left(  S_{q}\cup S_{r}\right)  \cap S_{2}$. Applying
the same reasoning to $s_{1}$, we see that the Evader must have a move into
some $s_{2}\in\left(  S_{q}\cup S_{r}\right)  \cap S_{1}$. Continuing in this
manner, we see that
\[
\exists\widetilde{\sigma}^{2}:\forall\sigma^{1}:H\left(  \sigma^{1}%
,\widetilde{\sigma}^{2}|s_{0}\right)  =s_{0}s_{1}s_{2}...\text{ such that:
}\forall n:s_{n}\in S_{q}\cup S_{r}\text{ }%
\]
which, since $S_{c}\cap\left(  S_{q}\cup S_{r}\right)  =\emptyset$, means
\[
\exists\widetilde{\sigma}^{2}:\forall\sigma^{1}:T\left(  s_{0}|\sigma
^{1},\widetilde{\sigma}^{2}\right)  =\infty
\]
which contradicts $s_{0}\in S_{q}$. Hence there cannot exist any $s_{0}\in
S_{q}\cap S^{1}$. Similarly we prove that there cannot exist any $s_{0}\in
S_{q}\cap S^{2}$. Hence $S_{q}$ is empty and we have proved (\ref{eq05004}).

We prove (\ref{eq05005}) immediately by using (\ref{eq05002}):%
\[
\left(  \forall\sigma^{1}:T\left(  \sigma^{1},\overline{\sigma}^{2}|s\right)
=\infty\right)  \Rightarrow\inf_{\sigma^{1}}T\left(  \sigma^{1},\overline
{\sigma}^{2}|s\right)  =\infty.
\]
The proof of the proposition is complete.\bigskip
\end{proof}

\bigskip

\bigskip

\noindent Our next goal is to prove that every $\left(  \mathbf{N},S^{1}%
,S^{2},S_{c},s\right)  $ has a value and optimal Pursuer and Evader
strategies. Before proceeding in this direction we need some auxiliary material.

\begin{definition}
\label{prop03010}Given the game $\left(  \mathbf{N},S^{1},S^{2},S_{c}%
,s\right)  $ and any $K\in$ $\mathbb{N}_{0}$, the $K$\emph{-truncated game}
(or simply the \emph{truncated game}) $\left(  \mathbf{N},S^{1},S^{2}%
,S_{c},s\right)  _{\left[  K\right]  }$ is identical to $\left(
\mathbf{N},S^{1},S^{2},S_{c},s\right)  $ except for the fact that it is played
for $K$ turns; consequently it has payoff function%
\[
Q\left(  s_{0}s_{1}...s_{K}\right)  =\sum_{t=0}^{K}q\left(  s_{t}\right)
\text{.}%
\]

\end{definition}

\begin{proposition}
\label{prop03011}For every $K\in$ $\mathbb{N}_{0}$, the game $\left(
\mathbf{N},S^{1},S^{2},S_{c},s\right)  _{\left[  K\right]  }$ has a value and
deterministic optimal strategies.
\end{proposition}

\begin{proof}
A full proof can be found in \cite{Thomas}. The main point is that the
truncated game is a \emph{finite} game and hence known to have a value and
optimal strategies. The strategies are deterministic because the game has
perfect information.
\end{proof}

\begin{definition}
\label{prop03012}Given the truncated game $\left(  \mathbf{N},S^{1}%
,S^{2},S_{c},s\right)  _{\left[  K\right]  }$ and strategies $\sigma
^{1},\sigma^{2}$, we define
\[
T_{\left[  K\right]  }\left(  \sigma^{1},\sigma^{2}|s\right)  =Q\left(
H\left(  \sigma^{1},\sigma^{2}|s\right)  \right)  .
\]
In other words, $T_{\left[  K\right]  }\left(  \sigma^{1},\sigma^{2}|s\right)
$ is $P^{2}$'s payoff in $\left(  \mathbf{N},S^{1},S^{2},S_{c},s\right)
_{\left[  K\right]  }$ when the strategies $\sigma^{1},\sigma^{2}$ are used.
\end{definition}

\bigskip

\noindent We use the bound $T_{a}\left(  s\right)  $ of Proposition
\ref{prop03009} to define the truncated game $\left(  \mathbf{N},S^{1}%
,S^{2},S_{c},s\right)  _{\left[  T_{a}\left(  s\right)  \right]  }$; by
Definition \ref{prop03012}, when the players use strategies $\sigma^{1}$ and
$\sigma^{2}$ the corresponding payoff is $T_{\left[  T_{a}\left(  s\right)
\right]  }\left(  \sigma^{1},\sigma^{2}|s\right)  $. To simplify notation we
write $\left(  \mathbf{N},S^{1},S^{2},S_{c},s\right)  _{\#}$ in place of
$\left(  \mathbf{N},S^{1},S^{2},S_{c},s\right)  _{\left[  T_{a}\left(
s\right)  \right]  }$ and $T_{\#}\left(  \sigma^{1},\sigma^{2}|s\right)  $ in
place of $T_{\left[  T_{a}\left(  s\right)  \right]  }\left(  \sigma
^{1},\sigma^{2}|s\right)  $. As mentioned, $\left(  \mathbf{N},S^{1}%
,S^{2},S_{c},s\right)  _{\#}$ has a value $\widehat{T}\left(  s\right)  $ and
an optimal deterministic strategy pair $\left(  \widehat{\sigma}^{1}%
,\widehat{\sigma}^{2}\right)  $. \ 

In what follows we will often need to use the \textquotedblleft
same\textquotedblright\ strategies in both $\left(  \mathbf{N},S^{1}%
,S^{2},S_{c},s\right)  $ and $\left(  \mathbf{N},S^{1},S^{2},S_{c},s\right)
_{\#}$. The domain of a strategy $\sigma^{n}$ for $\left(  \mathbf{N}%
,S^{1},S^{2},S_{c},s\right)  $ is the set of all finite histories, so we can
also use $\sigma^{n}$ in $\left(  \mathbf{N},S^{1},S^{2},S_{c},s\right)
_{\#}$ by applying it only to histories of length at most $T_{a}\left(
s\right)  $. To use a $\left(  \mathbf{N},S^{1},S^{2},S_{c},s\right)  _{\#}$
strategy $\sigma^{n}$ in $\left(  \mathbf{N},S^{1},S^{2},S_{c},s\right)  $, we
extend it as follows:\ for each history $h$ with length greater than
$T_{a}\left(  s\right)  $, $\sigma^{n}\left(  h\right)  $ is the stay-in-place move.

\bigskip

\noindent We are now ready to prove that $\left(  \mathbf{N},S^{1},S^{2}%
,S_{c},s\right)  $ has a value $\widehat{T}\left(  s\right)  $ and an optimal
deterministic strategy pair $\left(  \widehat{\sigma}^{1},\widehat{\sigma}%
^{2}\right)  $.

\begin{proposition}
\label{prop03013}For every $s\in\overline{S}$, the game $\left(
\mathbf{N},S^{1},S^{2},S_{c},s\right)  $ has a value $\widehat{T}\left(
s\right)  $ and optimal Pursuer and Evader strategies $\widehat{\sigma}^{1}$,
$\widehat{\sigma}^{2}$ which attain the value, i.e., $\widehat{T}\left(
s\right)  =T\left(  \widehat{\sigma}^{1},\widehat{\sigma}^{2}|s\right)  $.
\end{proposition}

\begin{proof}
We treat the P-win and E-win cases separately. Because of perfect information,
all strategies mentioned in the rest of the proof are \emph{deterministic}.

\noindent\underline{\textbf{Part I}}. Suppose that $s\in\overline{S}$ is such
that $\left(  \mathbf{N},S^{1},S^{2},S_{c},s\right)  $ is P-win. Keeping in
mind the stay-in-place extension of strategies we see the following.

\begin{enumerate}
\item $P^{1}$ can use the $\overline{\sigma}^{1}$ of Proposition
\ref{prop03009} in $\left(  \mathbf{N},S^{1},S^{2},S_{c},s\right)  _{\#}$ and
clearly we have
\[
\sup_{\sigma^{2}}T_{\#}\left(  \overline{\sigma}^{1},\sigma^{2}|s\right)  \leq
T_{a}\left(  s\right)
\]
(the Evader cannot have a higher payoff than the number of turns).

\item $P^{1}$ (resp. $P^{2}$) can use $\widehat{\sigma}^{1}$ (resp.
$\widehat{\sigma}^{2}$) in $\left(  \mathbf{N},S^{1},S^{2},S_{c},s\right)
_{\#}$ and
\begin{equation}
T\left(  \widehat{\sigma}^{1},\widehat{\sigma}^{2}|s\right)  =T_{\#}\left(
\widehat{\sigma}^{1},\widehat{\sigma}^{2}|s\right)  =\sup_{\sigma^{2}}%
\inf_{\sigma^{1}}T_{\#}\left(  \sigma^{1},\sigma^{2}|s\right)  =\inf
_{\sigma^{1}}\sup_{\sigma^{2}}T_{\#}\left(  \sigma^{1},\sigma^{2}|s\right)
\leq T_{a}\left(  s\right)  . \label{eq03029}%
\end{equation}
The equality $T\left(  \widehat{\sigma}^{1},\widehat{\sigma}^{2}|s\right)
=T_{\#}\left(  \widehat{\sigma}^{1},\widehat{\sigma}^{2}|s\right)  $ follows
from the fact that $\left(  \mathbf{N},S^{1},S^{2},S_{c},s\right)  $ is P-win;
hence, by Proposition \ref{prop03009}, $\overline{\sigma}^{1}$ guarantees
capture in at most $T_{a}\left(  s\right)  $ turns and $\widehat{\sigma}^{1}$
(which is optimal) will do at least as well. Since $\left(  \mathbf{N}%
,S^{1},S^{2},S_{c},s\right)  _{\#}$ lasts for $T_{a}\left(  s\right)  $ turns,
the required inequality is obvious.

\item From optimality and Proposition \ref{prop02017} we also get%
\begin{equation}
\forall\sigma^{1},\sigma^{2}:T_{\#}\left(  \widehat{\sigma}^{1},\sigma
^{2}|s\right)  \leq T_{\#}\left(  \widehat{\sigma}^{1},\widehat{\sigma}%
^{2}|s\right)  \leq T_{\#}\left(  \sigma^{1},\widehat{\sigma}^{2}|s\right)  .
\label{eq05006}%
\end{equation}

\end{enumerate}

Now suppose there exists some strategy $\widetilde{\sigma}^{1}$ such that
\[
T\left(  \widetilde{\sigma}^{1},\widehat{\sigma}^{2}|s\right)  <T\left(
\widehat{\sigma}^{1},\widehat{\sigma}^{2}|s\right)  \leq T_{a}\left(
s\right)  .
\]
Then we will also have
\[
T_{\#}\left(  \widetilde{\sigma}^{1},\widehat{\sigma}^{2}|s\right)  =T\left(
\widetilde{\sigma}^{1},\widehat{\sigma}^{2}|s\right)  <T\left(  \widehat
{\sigma}^{1},\widehat{\sigma}^{2}|s\right)  =T_{\#}\left(  \widehat{\sigma
}^{1},\widehat{\sigma}^{2}|s\right)
\]
which contradicts (\ref{eq05006}).

Similarly, suppose there exists some strategy $\widetilde{\sigma}^{2}$ such
that
\begin{equation}
T\left(  \widehat{\sigma}^{1},\widehat{\sigma}^{2}|s\right)  <T\left(
\widehat{\sigma}^{1},\widetilde{\sigma}^{2}|s\right)  . \label{eq05007}%
\end{equation}
Then we have the following cases.

\begin{enumerate}
\item If $T\left(  \widehat{\sigma}^{1},\widetilde{\sigma}^{2}|s\right)  \leq
T_{a}\left(  s\right)  $ then (both $\left(  \widehat{\sigma}^{1}%
,\widetilde{\sigma}^{2}\right)  $ and $\left(  \widehat{\sigma}^{1}%
,\widehat{\sigma}^{2}\right)  $ result in capture before the truncated game is
over):%
\[
T_{\#}\left(  \widehat{\sigma}^{1},\widehat{\sigma}^{2}|s\right)  =T\left(
\widehat{\sigma}^{1},\widehat{\sigma}^{2}|s\right)  <T\left(  \widehat{\sigma
}^{1},\widetilde{\sigma}^{2}|s\right)  =T_{\#}\left(  \widehat{\sigma}%
^{1},\widetilde{\sigma}^{2}|s\right)  ;
\]
which contradicts (\ref{eq05006}).

\item If $T\left(  \widehat{\sigma}^{1},\widetilde{\sigma}^{2}|s\right)
>T_{a}\left(  s\right)  $ then $T_{\#}\left(  \widehat{\sigma}^{1}%
,\widetilde{\sigma}^{2}|s\right)  =T_{a}\left(  s\right)  $ (the truncated
game finishes before capture)\ and we have the following subcases.
\[
T_{a}\left(  s\right)  >T\left(  \widehat{\sigma}^{1},\widehat{\sigma}%
^{2}|s\right)  \Rightarrow T_{\#}\left(  \widehat{\sigma}^{1},\widetilde
{\sigma}^{2}|s\right)  =T_{a}\left(  s\right)  >T_{\#}\left(  \widehat{\sigma
}^{1},\widehat{\sigma}^{2}|s\right)
\]
which contradicts (\ref{eq05006}); and
\[
T_{a}\left(  s\right)  =T\left(  \widehat{\sigma}^{1},\widehat{\sigma}%
^{2}|s\right)  \Rightarrow T_{\#}\left(  \widehat{\sigma}^{1},\widetilde
{\sigma}^{2}|s\right)  =T_{a}\left(  s\right)  =T_{\#}\left(  \widehat{\sigma
}^{1},\widehat{\sigma}^{2}|s\right)
\]
which contradicts (\ref{eq05007}).
\end{enumerate}

\noindent\underline{\textbf{Part II}}. Suppose that $s\in\overline{S}$ is such
that $\left(  \mathbf{N},S^{1},S^{2},S_{c},s\right)  $ is E-win. In this case
we start from (\ref{eq05005}) and get%
\[
\inf_{\sigma^{1}}T\left(  \sigma^{1},\overline{\sigma}^{2}|s\right)
=\infty\Rightarrow\sup_{\sigma^{2}}\inf_{\sigma^{1}}T\left(  \sigma^{1}%
,\sigma^{2}|s\right)  =\infty\Rightarrow\sup_{\sigma^{2}}\inf_{\sigma^{1}%
}T\left(  \sigma^{1},\sigma^{2}|s\right)  =\inf_{\sigma^{1}}\sup_{\sigma^{2}%
}T\left(  \sigma^{1},\sigma^{2}|s\right)  =\infty.
\]
Hence we have shown \ $\widehat{T}\left(  s\right)  =\infty$. Since
$\overline{\sigma}^{2}$ achieves $\widehat{T}\left(  s\right)  $, it is an
Evader optimal strategy. Any $\sigma^{1}$ achieves $\widehat{T}\left(
s\right)  $, hence any $\sigma^{1}$ is a Pursuer optimal strategy.
\end{proof}

\begin{definition}
\label{prop03014}We define the following two \emph{positional} strategies%
\begin{align*}
\text{A Pursuer strategy }\widehat{\sigma}^{1}  &  :\forall s\in
S^{1}\backslash S_{c}:\widehat{\sigma}^{1}\left(  s\right)  =\arg\min
{}_{s^{\prime}\in N\left(  s\right)  }\widehat{T}\left(  s^{\prime}\right)
,\\
\text{An Evader strategy }\widehat{\sigma}^{2}  &  :\forall s\in
S^{2}\backslash S_{c}:\widehat{\sigma}^{2}\left(  s\right)  =\arg\max
{}_{s^{\prime}\in N\left(  s\right)  }\widehat{T}\left(  s^{\prime}\right)  .
\end{align*}

\end{definition}

\noindent The next proposition shows $\left(  \widehat{\sigma}^{1}%
,\widehat{\sigma}^{2}\right)  $ is an optimal strategy pair in $\left(
\mathbf{N},S^{1},S^{2},S_{c},s\right)  $ \emph{for every} $s$. For the sake of
brevity, in the proof we will use the following notation:\ given a positional
strategy profile $\sigma=\left(  \sigma^{1},\sigma^{2}\right)  $ and a state
$s=\left(  x^{1},x^{2},p\right)  $, we let%
\[
\sigma\left(  s\right)  =\left(  \sigma^{p}\left(  s\right)  ,x^{-p}%
,-p\right)  .
\]
In other words, $\sigma\left(  s\right)  $ is the state to which $s$ transits
when the player $P^{p}$ (who has the move in $s$) applies his strategy
$\sigma^{p}$.

\begin{proposition}
\label{prop03015}For a given $\left(  \mathbf{N},S^{1},S^{2},S_{c},s\right)  $
and for all $s$, we have%
\begin{align}
\forall\sigma^{1}  &  :\forall s:T\left(  \widehat{\sigma}^{1},\widehat
{\sigma}^{2}|s\right)  \leq T\left(  \sigma^{1},\widehat{\sigma}^{2}|s\right)
,\label{eq06001}\\
\forall\sigma^{2}  &  :\forall s:T\left(  \widehat{\sigma}^{1},\widehat
{\sigma}^{2}|s\right)  \geq T\left(  \widehat{\sigma}^{1},\sigma^{2}|s\right)
. \label{eq06002}%
\end{align}

\end{proposition}

\begin{proof}
We will only prove (\ref{eq06001}) for any $s\in S^{1}$ (the proofs of the
remaining parts are similar). Let $s_{0}=s$ and $\widehat{\sigma}=\left(
\widehat{\sigma}^{1},\widehat{\sigma}^{2}\right)  $; furthermore, pick any
Pursuer strategy $\sigma^{1}$ and let $\widetilde{\sigma}=\left(  \sigma
^{1},\widehat{\sigma}^{2}\right)  $. Now define the state sequences
$s_{0}\widetilde{s}_{1}\widetilde{s}_{2}...$ and $s_{0}\widehat{s}_{1}%
\widehat{s}_{2}...$ as follows%
\[%
\begin{array}
[c]{lllll}%
\widetilde{s}_{1}=\widetilde{\sigma}\left(  s_{0}\right)  , & \widetilde
{s}_{2}=\widetilde{\sigma}\left(  \widetilde{s}_{1}\right)  , & \widetilde
{s}_{3}=\widetilde{\sigma}\left(  \widetilde{s}_{2}\right)  , & \widetilde
{s}_{4}=\widetilde{\sigma}\left(  \widetilde{s}_{3}\right)  , & ...\\
\widehat{s}_{1}=\widehat{\sigma}\left(  s_{0}\right)  , & \widehat{s}%
_{2}=\widehat{\sigma}\left(  \widetilde{s}_{1}\right)  , & \widehat{s}%
_{3}=\widehat{\sigma}\left(  \widetilde{s}_{2}\right)  , & \widehat{s}%
_{4}=\widehat{\sigma}\left(  \widetilde{s}_{3}\right)  , & ...
\end{array}
\]
Note that $\widetilde{s}_{0}\widetilde{s}_{1}\widetilde{s}_{2}...$ is
$H\left(  \widetilde{\sigma}|s_{0}\right)  $ but $\widehat{s}_{0}\widehat
{s}_{1}\widehat{s}_{2}...$ is \emph{not} $H\left(  \widehat{\sigma}%
|s_{0}\right)  $ (why?). Also note that%
\[
\widetilde{s}_{2}=\widehat{s}_{2},\quad\widetilde{s}_{4}=\widehat{s}_{4}%
,\quad\widetilde{s}_{6}=\widehat{s}_{6},\quad\widetilde{s}_{8}=\widehat{s}%
_{8},\quad...\text{ .}%
\]
Because \ $\widehat{\sigma}^{1}$ always chooses minimizing successor states,
we have the following sequence of inequalities.%
\begin{align*}
T\left(  \widehat{\sigma}|s_{0}\right)   &  =1+T\left(  \widehat{\sigma
}|\widehat{s}_{1}\right)  \leq1+T\left(  \widehat{\sigma}|\widetilde{s}%
_{1}\right) \\
T\left(  \widehat{\sigma}|\widetilde{s}_{1}\right)   &  =1+T\left(
\widehat{\sigma}|\widehat{s}_{2}\right)  =1+T\left(  \widehat{\sigma
}|\widetilde{s}_{2}\right)  \Rightarrow T\left(  \widehat{\sigma}%
|s_{0}\right)  \leq2+T\left(  \widehat{\sigma}|\widetilde{s}_{2}\right) \\
T\left(  \widehat{\sigma}|\widetilde{s}_{2}\right)   &  =1+T\left(
\widehat{\sigma}|\widehat{s}_{3}\right)  \leq1+T\left(  \widehat{\sigma
}|\widetilde{s}_{3}\right)  \Rightarrow T\left(  \widehat{\sigma}%
|s_{0}\right)  \leq3+T\left(  \widehat{\sigma}|\widetilde{s}_{3}\right) \\
T\left(  \widehat{\sigma}|\widetilde{s}_{3}\right)   &  =1+T\left(
\widehat{\sigma}|\widehat{s}_{4}\right)  =1+T\left(  \widehat{\sigma
}|\widetilde{s}_{4}\right)  \Rightarrow T\left(  \widehat{\sigma}%
|s_{0}\right)  \leq4+T\left(  \widehat{\sigma}|\widetilde{s}_{4}\right) \\
&  ...
\end{align*}
This sequence of inequalities (i)\ either will continue until we reach some
$\widetilde{s}_{K}\in S_{c}$ (ii)\ or, if no capture state is ever reached,
will go on ad infinitum. Let us consider each case separately.

\begin{enumerate}
\item If there exists some $K\in\mathbb{N}$ such that $\widetilde{s}_{K}\in
S_{c}$ , then $T\left(  \widetilde{\sigma}|s_{0}\right)  =K\ $and we have%
\[
T\left(  \widehat{\sigma}|s_{0}\right)  \leq K+T\left(  \widehat{\sigma}%
|s_{K}\right)  =T\left(  \widetilde{\sigma}|s_{0}\right)  +0.
\]

\item If for all $k\in\mathbb{N}$ we have $s_{k}\not \in S_{c}$ , then
$T\left(  \widetilde{\sigma}|s_{0}\right)  =\infty\ $and we have%
\[
T\left(  \widehat{\sigma}|s_{0}\right)  \leq\infty=T\left(  \widetilde{\sigma
}|s_{0}\right)  .
\]

\end{enumerate}

\noindent In either case we have proved that $T\left(  \widehat{\sigma}%
|s_{0}\right)  \leq T\left(  \widetilde{\sigma}|s_{0}\right)  $ which, written
in more detail, is%
\[
\forall\sigma^{1}:\forall s_{0}\in S^{1}:T\left(  \widehat{\sigma}%
^{1},\widehat{\sigma}^{2}|s_{0}\right)  \leq T\left(  \sigma^{1}%
,\widehat{\sigma}^{2}s_{0}\right)  .
\]
In other words we have proved (\ref{eq06001}) when $s\in S^{1}$; the proof of
(\ref{eq06001}) when $s\in S^{2}$ as well as the proof of (\ref{eq06002}) are similar.
\end{proof}

To summarize, up to this point we have proved that $\left(  \mathbf{N}%
,S^{1},S^{2},S_{c},s\right)  $ has a value and optimal positional strategies
$\widehat{\sigma}^{1},\widehat{\sigma}^{2}$\ (specified in terms of the
collection of values $\left(  \widehat{T}\left(  s\right)  \right)
_{s\in\overline{S}}$)\ \emph{without using the VL\ algorithm of Section
}\ref{sec0301}. Our final target is to show, \emph{without using the results
of Section \ref{sec0301}}, that the values (i)\ satisfy the optimality
equations and (ii)\ can be computed by the VL\ algorithm.

\begin{proposition}
\label{prop03016}The values $\left(  \widehat{T}\left(  s\right)  \right)
_{s\in\overline{S}}$ of the games $\left(  \mathbf{N},S^{1},S^{2}%
,S_{c},s\right)  _{s\in\overline{S}}$ satisfy the \emph{optimality equations}.%
\begin{align}
\forall s  &  \in S_{c}:\widehat{T}\left(  s\right)  =0,\\
\forall s  &  \in S^{1}\backslash S_{c}:\widehat{T}\left(  s\right)
=1+\min_{s^{\prime}\in N\left(  s\right)  }\widehat{T}\left(  s^{\prime
}\right)  ,\\
\forall s  &  \in S^{2}\backslash S_{c}:\widehat{T}\left(  s\right)
=1+\max_{s^{\prime}\in N\left(  s\right)  }\widehat{T}\left(  s^{\prime
}\right)  .
\end{align}

\end{proposition}

\begin{proof}
Straightforward.
\end{proof}

\begin{proposition}
\label{prop03017}For every $n\in\mathbb{N}_{0}$we have%
\begin{equation}
\forall s\in\overline{S}:\widehat{T}\left(  s\right)  =n\Rightarrow\left\{
\begin{array}
[c]{l}%
\forall m<n:T^{m}\left(  s\right)  =\infty\\
\forall m\geq n:T^{n}\left(  s\right)  =\widehat{T}\left(  s\right)
\end{array}
\right.  \label{eq01}%
\end{equation}
where $\left(  \left(  T^{n}\left(  s\right)  \right)  _{s\in\overline{S}%
}\right)  _{n\in\mathbb{N}_{0}}$ are the quantities computed by the VL\ algorithm.
\end{proposition}

\begin{proof}
Clearly (\ref{eq01}) holds for $n=0$. Suppose it holds for all $n\in\left\{
1,2,...,k\right\}  $ and consider two cases.

\begin{enumerate}
\item Take any $s\in S^{1}$ such that $\widehat{T}\left(  s\right)  =k+1$.
Then
\[
\widehat{T}\left(  s\right)  =1+\min_{s^{\prime}\in N\left(  s\right)
}\widehat{T}\left(  s^{\prime}\right)  .
\]
Hence%
\[
\exists s_{1}\in N\left(  s\right)  :\widehat{T}\left(  s_{1}\right)  =k\text{
and }\not \exists s_{2}\in N\left(  s\right)  :\widehat{T}\left(
s_{2}\right)  <k\text{. }%
\]
Then from the inductive hypothesis it follows that%
\begin{align}
\exists s_{1}  &  \in N\left(  s\right)  :T^{k}\left(  s_{1}\right)
=k\text{,}\label{eq02}\\
\not \exists s_{2}  &  \in N\left(  s\right)  :T^{k}\left(  s_{2}\right)
<k\text{. } \label{eq03}%
\end{align}
From (\ref{eq02})-(\ref{eq03}) we have
\begin{equation}
k=\min_{s^{\prime}\in N\left(  s\right)  }T^{k}\left(  s^{\prime}\right)
\Rightarrow T^{k+1}\left(  s\right)  =k+1. \label{eq03040a}%
\end{equation}
From (\ref{eq03}) we also have
\begin{equation}
\left(  \forall s^{\prime}\in N\left(  s\right)  :T^{0}\left(  s^{\prime
}\right)  =...=T^{k-1}\left(  s^{\prime}\right)  =\infty\right)  \Rightarrow
T^{k}\left(  s\right)  =1+\infty=\infty. \label{eq03041}%
\end{equation}
From (\ref{eq03040a})-(\ref{eq03041})\ and lines 10-16 of the VL\ algorithm we
see that (\ref{eq01})\ holds for all $s\in S^{1}$ such that $\widehat
{T}\left(  s\right)  =k+1$.

\item Take any $s\in S^{2}$ such that $\widehat{T}\left(  s\right)  =k+1$.
Then
\[
\widehat{T}\left(  s\right)  =1+\max_{s^{\prime}\in N_{out}\left[  s\right]
}\widehat{T}\left(  s^{\prime}\right)  .
\]
Hence%
\[
\exists s_{1}\in N\left(  s\right)  :\widehat{T}\left(  s_{1}\right)  =k\text{
and }\forall s^{\prime}\in N\left(  s\right)  :\widehat{T}\left(  s^{\prime
}\right)  \leq k\text{. }%
\]
Then from the inductive hypothesis it follows that%
\begin{align}
\exists s_{1}  &  \in N\left(  s\right)  :T^{k}\left(  s_{1}\right)
=k\text{,}\label{eq04}\\
\forall s^{\prime}  &  \in N\left(  s\right)  :T^{k}\left(  s^{\prime}\right)
\leq k\text{. } \label{eq05}%
\end{align}
From (\ref{eq04})-(\ref{eq05}) we have
\begin{equation}
k=\max_{s^{\prime}\in N\left(  s\right)  }T^{k}\left(  s^{\prime}\right)
\Rightarrow T^{k+1}\left(  s\right)  =k+1. \label{eq03042}%
\end{equation}
From (\ref{eq04}) we also have
\begin{equation}
T^{k-1}\left(  s_{1}\right)  =\infty\Rightarrow T^{k}\left(  s\right)
=1+\infty=\infty. \label{eq03043}%
\end{equation}
From (\ref{eq03042})-(\ref{eq03043})\ and lines 10-16 of the VL\ algorithm we
see that (\ref{eq01})\ holds for all $s$ $\in S^{2}$ such that $\widehat
{T}\left(  s\right)  =k+1$.
\end{enumerate}

Hence (\ref{eq01})\ holds for all $s$ $\in\overline{S}$ such that $\widehat
{T}\left(  s\right)  =k+1$ and the proof is completed.
\end{proof}

\section{Comparison to Other Approaches\label{sec04}}

Our analysis of Section \ref{sec0301} is heavily inspired by
\cite{Hahn2006,Bonato2017}. We will now discuss these two papers and also the
less well known \cite{Berarducci1993} in comparison to our own. Let us
emphasize that, while we will criticize some aspects of
\cite{Hahn2006,Bonato2017,Berarducci1993}, we find these papers extremely
useful; they have provided the inspiration and foundation for our own more
detailed approach.

\subsection{Hahn and MacGillivray}

In \cite{Hahn2006} Hahn and MacGillivray study a CR\ version with two
generalizations of the classic game:\ (i)\ the game is played on a directed
graph and (ii)\ more than one cops and/or robbers (\textquotedblleft$k$-cop,
$l$-robber\textquotedblright)\ may be involved\footnote{Let us stress that
they still deal with a two-player game:\ there is a single Cop player and a
single Robbber player, but each can control one or more cop and robber
\emph{tokens}.}. On the other hand, following the classic CR\ formulation,
they count time (especially capture time) in \emph{rounds}; one round includes
one move by each cop and robber token. While we consider the single Cop and
single Robber case, our own formulation can easily accommodate all of the above.

Next we describe two more substantial differences between Hahn and
MacGillivray's approach and our own. These are really differences between the
games being studied in each case. Namely, in \cite{Hahn2006}:

\begin{enumerate}
\item it is assumed that the two players move alternately and the game always
starts with the Cop moving first (once again this follows the classic CR game formulation);

\item the game starts with an empty graph, the Cop's first move is to place
his token on some vertex and the Robber's first move is to place his own
token; these two moves constitute the \textquotedblleft\emph{placement
round}\textquotedblright.
\end{enumerate}

However both of the above differences can be easily accommodated by our
approach. Obviously, removing the \textquotedblleft alternating
moves\textquotedblright\ assumption makes our analysis more general. To
accommodate the \textquotedblleft placement round\textquotedblright, we can
use a one-round game which consists of two turns:\ first the Cop chooses a
vertex $x_{0}^{1}$, then the Robber chooses a vertex $x_{0}^{2}$ and then the
\ Robber gains (the Cop loses) $\widehat{T}\left(  x_{0}^{1},x_{0}%
^{2},1\right)  $\ payoff units, where $\widehat{T}\left(  x_{0}^{1},x_{0}%
^{2},1\right)  $\ has been computed for every vertex pair $\left(  x_{0}%
^{1},x_{0}^{2}\right)  $ by the VL algorithm. Clearly the new game has a value
which is
\[
\min_{x_{0}^{1}}\max_{x_{0}^{2}}\widehat{T}\left(  x_{0}^{1},x_{0}%
^{2},1\right)  .
\]
Hence the solution of our GCR\ game $\ $ also provides the solution to Hahn
and MacGillivray's (classic) CR\ game.

One of the main components of \cite{Hahn2006} is a vertex labeling \ algorithm
very similar to our own, which is used to compute optimal capture times
(counted in \emph{rounds}) and strategies. The main properties of this
algorithm are established in Lemma 4 of \cite{Hahn2006}. We find the proof of
this Lemma not quite rigorous, because precise definitions of \emph{strategy}
and \emph{optimality} (and also \emph{value}) are not provided.

Both \textquotedblleft\emph{strategy}\textquotedblright\ and \textquotedblleft%
\emph{optimal strategy}\textquotedblright\ are used informally in
\cite{Hahn2006}. \textquotedblleft\emph{Strategy}\textquotedblright\ is not
defined. \textquotedblleft\emph{Optimal strategy}\textquotedblright\ for the
Cop is defined informally as follows:\ \textquotedblleft a strategy from a
configuration $c_{xy}$ [is] \emph{optimal} for the cop if no other strategy
gives a win in fewer moves\textquotedblright. But this is incorrect (indeed
there exist Cop strategies which give capture time \emph{better than optimal},
but only for \emph{some} robber strategies) because it does not take into
account the Robber's strategy. Similar remarks can be made regarding the
informal definition of optimal Robber strategies.

To precisely define \textquotedblleft\emph{optimal strategy}\textquotedblright%
\ one must first define \textquotedblleft\emph{strategy}\textquotedblright%
\ (as a function from histories to moves)\ and then provide an optimality
criterion. In Game Theory optimality is defined in connection with
\textquotedblleft game value\textquotedblright\ (which is neither defined nor
used in \cite{Hahn2006}). In the context of CR\ and GCR the appropriate
definitions are Definitions \ref{prop02016} and \ref{prop02017} as given in
our Section \ref{sec0203}.

Consequently, while the main ideas in the proof of Lemma 4 are correct, their
elaboration is not always rigorous (in our opinion). But in some cases the
authors' arguments can be improved quite easily. For instance, their statement
\textquotedblleft\emph{the cop's move will be to an }$x^{\prime}$\emph{ ...
from which, by the induction hypothesis, the cop can win in }$t-1$\emph{
rounds}\textquotedblright\ should be augmented by: \textquotedblleft\emph{no
matter how the robber plays}\textquotedblright\footnote{Of course this is just
a verbal description of the $\sup$ and $\inf$ conditions of our Definition
\ref{prop02016}.}. Similar remarks apply to other parts of \cite{Hahn2006}.

\subsection{Bonato and MacGillivray}

As already stated, our main inspiration is \cite{Bonato2017}, in which Bonato
and MacGillivray generalize the games and results of \cite{Hahn2006}. In place
of \textquotedblleft Cop\textquotedblright\ and \textquotedblleft
Robber\textquotedblright, they use the terms \textquotedblleft
Pursuer\textquotedblright\ and \textquotedblleft Evader\textquotedblright.
\textquotedblleft Alternating moves\textquotedblright\ and \textquotedblleft
placement round\textquotedblright\ are used in the same manner as in
\cite{Hahn2006}. On the other hand capture is understood in a more general
sense; slightly paraphrasing \cite{Bonato2017}, the Pursuer wins if, at any
time-step, the current position of the game belongs to the subset of
\emph{final positions}. Of course this is exactly analogous to our capture set
$S_{c}$.

A vertex labeling\ algorithm is also provided in \cite{Bonato2017}; it counts
time in turns (not rounds) and is essentially the same as our own
VL\ Algorithm\footnote{There is one caveat:\ it is never specified in
\cite{Bonato2017} whether once a state achieves a finite label can be
subsequently relabeled (it should not); this is probably an oversight in the
description.}. However, rather than proving directly the properties of their
algorithm, the authors proceed in the following manner.

\begin{enumerate}
\item The construct, independently of the labeling algorithm, a sequence of
orderings $\preceq_{0}$, $\preceq_{1}$, ... on Pursuer and Evader
\emph{positions}.

\item They prove that these converge to an ordering $\preceq$.

\item They relate winning and \textquotedblleft optimal\textquotedblright%
\ game duration to $\left(  \preceq_{i}\right)  _{i\in\mathbb{N}_{0}}$ and
$\preceq$ (their Theorem 3.1 and Corollary 3.2).

\item Finally they relate state labels to the orderings $\left(  \preceq
_{i}\right)  _{i\in\mathbb{N}_{0}}$ (Theorem 3.3)
\end{enumerate}

Hence their vertex labeling\ \ algorithm is peripheral, rather than central to
the arguments of \cite{Bonato2017}. Nevertheless, the criticisms addressed to
\cite{Hahn2006} can also be addressed to \cite{Bonato2017}. Namely,
\textquotedblleft strategy\textquotedblright, \textquotedblleft
value\textquotedblright\ and \textquotedblleft optimality\textquotedblright%
\ are used but not rigorously defined. An informal definition of optimality is
that an \textquotedblleft the Pursuer's optimal strategy is to move so that
the game is over as quickly as possible, and the Evader's optimal strategy is
to move so the game lasts as long as possible\textquotedblright; similarly to
\cite{Hahn2006}, this definition does not clarify the role of the
\textquotedblleft other\textquotedblright\ player's strategy. A correct verbal
description would be:\ \textquotedblleft the Pursuer's optimal strategy is to
move so that the \emph{longest possible} duration of the game is as short as
possible\textquotedblright\ (and a similar modification should be appplied to
the definition of the Evader's optimal strategy). Now, the above are simply
verbal descriptions of the $\inf_{\sigma^{1}}\sup_{\sigma^{2}}$ and
$\sup_{\sigma^{2}}\inf_{\sigma^{1}}$ conditions on capture time \emph{and they
do not suffice to ensure optimality} (we must have in addition that
$\inf_{\sigma^{1}}\sup_{\sigma^{2}}$ equals $\sup_{\sigma^{2}}\inf_{\sigma
^{1}}$, as in Definition \ref{prop02016}).

\subsection{Berarducci and Intrigila}

The earliest investigation of optimal CR\ strategies that we know of is the
one presented in \cite{Berarducci1993} by Berarducci and Intrigila. As we will
explain below, this work provides a very useful approach to the CR\ problem.

Berarducci and Intrigila \emph{do} provide a definition of strategies, both
general and positional. In their Remark 2.2 they apparently assume implicitly
that an optimal solution can be found by considering only positional
strategies but they actually justify (post facto) this assumption.

Interestingly, the results of \cite{Berarducci1993} are established by using a
sequence of sets $W_{0}$, $W_{1}$, ... (rather than a labeling algorithm). The
sequence is defined inductively: their $W_{0}$ is our capture sets $S_{c}$
and, for each $n$, $W_{n}$ is defined (their Definition 2.4)\ in a manner
which strongly resembles our VL\ Algorithm. The subsequent arguments
(contained in their Lemmas 2.5-2.7)\ resemble the analysis of our
VL\ Algorithm. Their main results are the following.

\begin{enumerate}
\item The set $W_{n}$ is the set of all these starting states from which the
Cop \emph{can} capture the Robber in $n$ moves or less (their Lemma 2.5).

\item The sequence $W_{n}$ converges to a state set $W$ which has the
following property:\ for every starting state $s\in W$ the Cop \emph{can}
capture the Robber in a finite number of moves; for every starting state
$s\notin W$ the Cop \emph{cannot} capture the Robber in a finite number of
moves (their Lemma 2.6).

\item Optimal Cop and Robber strategies are also defined in the proof of Lemma
2.6 and they are, by their definition, positional.
\end{enumerate}

While in the above results Berarducci and Intrigila make no explicit mention
of the \textquotedblleft other player's\textquotedblright\ strategy, the use
of \textquotedblleft can\textquotedblright\ implies that \emph{the Cop has a
strategy which guarantees capture no matter how the Robber plays}. Similarly,
the use of \textquotedblleft cannot\textquotedblright\ implies that \emph{the
Robber has a strategy which guarantees noncapture no matter how the Cop plays}.

The proof of the above results is correct. In our understanding, the important
quantities are not the sets $W_{n}$ but the sets $U_{n}=W_{n}\backslash
W_{n-1}$. While not explicitly stated, it follows from their proof that
$U_{n}$ is the set of initial states from which

\begin{enumerate}
\item the Cop \emph{can} capture the Robber in \emph{at most }$n$ rounds,
\emph{no matter how the Robber plays};

\item but the Robber \emph{can} delay capture for \emph{at least} $n-1$
rounds, \emph{no matter how the Cop plays};
\end{enumerate}

In short, the analysis of \cite{Berarducci1993} respects, at least implicitly,
all the relevant game theoretic considerations. It is also closely related to
the previously mentioned vertex labeling algorithms. For example, it is easy
to prove that, reverting to our own terminology, the state $s$ belongs to
$U_{n}$ iff $\widehat{T}\left(  s\right)  =n$.

\section{Concluding Remarks\label{sec05}}

We have presented two full game theoretic solutions of the GCR\ game. Let us
briefly comment on each one.

Our first solution, presented in Section \ref{sec0301}, is self-contained and
follows closely \cite{Hahn2006} and \cite{Bonato2017}. The main tool for this
solution is the VL\ algorithm which is based on similar algorithms introduced
in \cite{Hahn2006,Bonato2017}. Our main contribution in Section \ref{sec0301}
is to present in greater detail and precision certain implicit assumptions of
\cite{Hahn2006,Bonato2017}.

Our second solution, presented in Section \ref{sec0302}, is also fully game
theoretic but is not totally self-contained (it invokes Von Neumann's MinMax
Theorem). It is also longer. The main reason we have presented it is that it
brings into the foreground certain aspects of the CR\ and GCR\ game which are
usually overlooked.

We conclude the current paper by listing (i)\ further generalizations of the
GCR\ game and (ii)\ well known families of games which contain GCR\ as a
special case.

\bigskip

\noindent\underline{\textbf{Generalizations of GCR}}. The GCR\ game, as
presented in both \cite{Bonato2017} and the current paper is a perfect
information, two-person, zero-sum game. All of these aspects can be generalized.

\begin{enumerate}
\item \underline{\emph{Concurrent GCR}}. A standard assumption of both classic
CR\ and Bonato and MacGillivray's GCR\ is that a single player moves in each
turn of the game. An obvious generalization is to allow both players to move
concurrently. In this case the game no longer has perfect information. It
still has a value which, however, will in genral be achieved by randomized
optimal strategies. An exploration in this direction appears in \cite{CCR}.

\item \underline{\emph{Nonzero-sum GCR}}. By modifying the payoff function we
can obtain a two-player non-zero sum CR\ game. For example, introducing
\textquotedblleft energy cost\textquotedblright, the Robber's payoff could be
the capture time minus the distance he has traveled and the Cop's payoff could
be the negative of the sum of capture time and the distance he has traveled.
An example of a two-player, nonzero-sum GCR\ game has been presented in
\cite{SCPR}; it involves two \emph{selfish }Cop players who attempt to catch a
\textquotedblleft passive\textquotedblright\ robber; by \textquotedblleft
passive\textquotedblright\ we mean that the robber is not controlled by a
player but follows, instead, a predetermined path, known to both cop players.

\item \underline{\emph{Multi-player GCR}}. Both the classic CR\ game and
practically all its published variants are \emph{two}-player, \emph{zero}%
-sum\ games; the same holds for the GCR of \cite{Bonato2017}. While such games
may involve more than one Pursuer, all Pursuers are controlled by a single
player whose payoff is given by single function. On the other hand, in
\cite{SCAR}\ we have studied an $N$-player (with $N\geq2$), \emph{nonzero}%
-sum\ version of the classic CR, the so-called \emph{Selfish Cops and
Adversarial Robber} (SCAR) game. As the name indicates, SCAR\ involves several
\emph{selfish} cops, each controlled by a separate player. All cop players
share the goal of catching the robber but each cop player has his own payoff
function which assigns a higher reward to the player who actually effects the
capture; the robber player wants, as in the classic CR\ game, to delay capture
as long as possible. We have generalized this approach in \cite{GenPurEv},
where we have introduced $N$\emph{-player Generalized CR\ Games}.
\end{enumerate}

\bigskip

\noindent\underline{\textbf{Additional Game Families}}

\begin{enumerate}
\item \underline{\emph{Stochastic Games}}. The above presented GCR\ games can
be formulated and as stochastic games. Using standard stochastic game results
\cite{filar1996} the following things can be shown for every GCR\ game:\ 

\begin{enumerate}
\item if it is zero-sum, it possesses a value and optimal strategies, which
can be computed by the \emph{Value Iteration Algorithm}, a generalization of
the VL\ algorithm presented in this paper;

\item if it is nonzero-sum, it possesses at least one \emph{Nash Equilibrium}
in deterministic positional strategies \cite{SCAR,GenPurEv}.
\end{enumerate}

\item \underline{\emph{Reachability games}}. The two-player, zero-sum
GCR\ game (understood in the sense of either the current paper or
\cite{Bonato2017}) can also be seen as a special type of \emph{reachability
game} \cite{Mazala,Berwanger}. In a reachability game the first player's
objective is to bring the game to a target state and the second player's
objective is to keep the game away from all target states. This is very
similar to the GCR\ game except that no assumption is made regarding the
players' locations. Indeed, a reachability game can be represented by a tuple
$\left(  \mathbf{N},S^{1},S^{2},S\right)  $ where $\mathbf{N}$ represents an
abstract collection of successor states. Our solution of the GCR\ game can be
applied to any reachability game. However the \textquotedblleft
usual\textquotedblright\ way to solve a reachability game is by constructing a
sequence of \emph{attractor sets}; this approach is practically identical to
the one used in \cite{Berarducci1993} to solve the classic CR\ game.

\item \underline{\emph{Graphical Games}}. Reachability games are perhaps the
simplest example of \emph{infinite perfect information games }%
\cite{Berwanger,Gradel,Mazala,Ummels} which can also be understood as games in
which two or more players move a token along the edges of a graph (hence the
term \textquotedblleft\emph{graphical games}\textquotedblright). Various
\emph{infinitary} winning conditions can be used which, in general, depend on
some property of the entire game history (for example:\ Player 1 wins if a
certain state is visited infinitely often). In the most general setting we can
have games with any number of players and nonzero-sum winning conditions.
\end{enumerate}

\newpage

\appendix

\section{Some Facts about Zero-Sum Games\label{secA}}

Here we present the general statements of several game theoretic definitions
and propositions (these were presented in GCR-specific form in Section
\ref{sec0203}). All of the following definitions and propositions refer to a
general two-player zero-sum game $\Gamma$ (i.e., they are not specific to the
games we discuss in the main body of the paper). The game is assumed to have a
payoff to $P^{2}$ (the \emph{maximizer}) equal to $U\left(  \sigma^{1}%
,\sigma^{2}\right)  $ where $\sigma^{n}$ is the strategy used by $P^{n}$
($n\in\left\{  1,2\right\}  $); the payoff to $P^{1}$ (the \emph{minimizer})
is equal to $-U\left(  \sigma^{1},\sigma^{2}\right)  $. Proofs can be found in
\cite{Maschler,Thomas}.

\begin{definition}
We define the following two quantities%
\begin{align*}
\text{lower value of }\Gamma &  :U^{-}\left(  s\right)  =\sup_{\sigma^{2}}%
\inf_{\sigma^{1}}U\left(  \sigma^{1},\sigma^{2}\right)  ,\\
\text{upper value of }\Gamma &  :U^{+}\left(  s\right)  =\inf_{\sigma^{1}}%
\sup_{\sigma^{2}}U\left(  \sigma^{1},\sigma^{2}\right)  .
\end{align*}

\end{definition}

\begin{proposition}
We always have
\[
U^{-}=\sup_{\sigma^{2}}\inf_{\sigma^{1}}U\left(  \sigma^{1},\sigma^{2}\right)
\leq\inf_{\sigma^{1}}\sup_{\sigma^{2}}U\left(  \sigma^{1},\sigma^{2}\right)
=U^{+}.
\]

\end{proposition}

\begin{definition}
We say that $\overline{\sigma}^{1}$ is \emph{minmax strategy} (for $P^{1}$)
iff
\[
\forall\sigma^{2}:U\left(  \overline{\sigma}^{1},\sigma^{2}\right)
=U^{+}=\inf_{\sigma^{1}}\sup_{\sigma^{2}}U\left(  \sigma^{1},\sigma
^{2}\right)  .
\]
We say that $\overline{\sigma}^{2}$ is \emph{maxmin strategy} (for $P^{2}$)
iff%
\[
\forall\sigma^{1}:U\left(  \sigma^{1},\overline{\sigma}^{2}\right)
=U^{-}=\sup_{\sigma^{2}}\inf_{\sigma^{1}}U\left(  \sigma^{1},\sigma
^{2}\right)  .
\]

\end{definition}

\begin{proposition}
For every minmax strategy $\overline{\sigma}^{1}$ we have%
\[
\forall\sigma^{1},\sigma^{2}:U\left(  \overline{\sigma}^{1},\sigma^{2}\right)
\leq\sup_{\sigma^{2}}U\left(  \sigma^{1},\sigma^{2}\right)  .
\]
For every maxmin strategy $\overline{\sigma}^{2}$ we have%
\[
\forall\sigma^{1},\sigma^{2}:U\left(  \sigma^{1},\overline{\sigma}^{2}\right)
\geq\inf_{\sigma^{1}}U\left(  \sigma^{1},\sigma^{2}\right)  .
\]

\end{proposition}

\begin{definition}
We say that $\Gamma$ has a \emph{value} $\widehat{U}\ $iff
\[
U^{-}=\sup_{\sigma^{2}}\inf_{\sigma^{1}}U\left(  \sigma^{1},\sigma^{2}\right)
=\inf_{\sigma^{1}}\sup_{\sigma^{2}}U\left(  \sigma^{1},\sigma^{2}\right)
=U^{+}.
\]
in which case we define%
\[
\widehat{U}=\sup_{\sigma^{2}}\inf_{\sigma^{1}}U\left(  \sigma^{1},\sigma
^{2}\right)  =\inf_{\sigma^{1}}\sup_{\sigma^{2}}U\left(  \sigma^{1},\sigma
^{2}\right)  .
\]

\end{definition}

\begin{definition}
If $\Gamma$ has a value $\widehat{U}\ $ and $\widehat{\sigma}^{1}$ (resp.
$\widehat{\sigma}^{2}$) is a minmax (resp. maxmin)\ strategy, then we call
$\widehat{\sigma}^{1}$ (resp. $\widehat{\sigma}^{2}$) an \emph{optimal
strategy} for $P^{1}$ (resp. for $P^{2}$) and we also say that $\left(
\widehat{\sigma}^{1},\widehat{\sigma}^{2}\right)  $ is an \emph{optimal
strategy pair}.
\end{definition}

\begin{proposition}
Every \emph{finite} (two player)\ zero-sum game has a value $\widehat{U}$ and
an optimal strategy pair $\left(  \widehat{\sigma}^{1},\widehat{\sigma}%
^{2}\right)  $, for which the following hold%
\[
\widehat{U}=U\left(  \widehat{\sigma}^{1},\widehat{\sigma}^{2}\right)
=\sup_{\sigma^{2}}\inf_{\sigma^{1}}U\left(  \sigma^{1},\sigma^{2}\right)
=\inf_{\sigma^{1}}\sup_{\sigma^{2}}U\left(  \sigma^{1},\sigma^{2}\right)  .
\]

\end{proposition}

\begin{definition}
We say that $\left(  \widehat{\sigma}^{1},\widehat{\sigma}^{2}\right)  $ is a
\emph{Nash Equilbrium } (NE)\ of $\Gamma$ iff%
\begin{align*}
\forall\sigma^{1}  &  :U\left(  \widehat{\sigma}^{1},\widehat{\sigma}%
^{2}\right)  \leq U\left(  \sigma^{1},\widehat{\sigma}^{2}\right)  ,\\
\forall\sigma^{2}  &  :U\left(  \widehat{\sigma}^{1},\widehat{\sigma}%
^{2}\right)  \geq U\left(  \widehat{\sigma}^{1},\sigma^{2}\right)  .
\end{align*}

\end{definition}

\noindent The next proposition says:\ $\left(  \widehat{\sigma}^{1}%
,\widehat{\sigma}^{2}\right)  $ is an optimal strategy pair iff it is a
$\emph{Nash\ Equilbrium}$.

\begin{proposition}
The following two conditions are equivalent.%
\begin{align}
&  U\left(  \widehat{\sigma}^{1},\widehat{\sigma}^{2}\right)  =\sup
_{\sigma^{2}}\inf_{\sigma^{1}}U\left(  \sigma^{1},\sigma^{2}\right)
=\inf_{\sigma^{1}}\sup_{\sigma^{2}}U\left(  \sigma^{1},\sigma^{2}\right)
\tag{\textbf{C1}}\\
&  \forall\sigma^{1}:U\left(  \widehat{\sigma}^{1},\widehat{\sigma}%
^{2}\right)  \leq U\left(  \sigma^{1},\widehat{\sigma}^{2}\right)  \text{ and
}\forall\sigma^{2}:U\left(  \widehat{\sigma}^{1},\widehat{\sigma}^{2}\right)
\geq U\left(  \widehat{\sigma}^{1},\sigma^{2}\right)  . \tag{\textbf{C2}}%
\end{align}

\end{proposition}

\begin{proof}
To prove that $\mathbf{C1}$ implies $\mathbf{C2}$ we note that, since
$\widehat{\sigma}^{1}$ (resp. $\widehat{\sigma}^{2}$)\ is optimal, it is also
a minmax (resp. maxmin)\ strategy. Hence from $\mathbf{C1}$ we have
\begin{align*}
\forall\sigma^{1},\sigma^{2}  &  :U\left(  \widehat{\sigma}^{1},\sigma
^{2}\right)  \leq\sup_{\sigma^{2}}U\left(  \sigma^{1},\sigma^{2}\right)
\Rightarrow\forall\sigma^{2}:U\left(  \widehat{\sigma}^{1},\sigma^{2}\right)
\leq\inf_{\sigma^{1}}\sup_{\sigma^{2}}U\left(  \sigma^{1},\sigma^{2}\right)
=U\left(  \widehat{\sigma}^{1},\widehat{\sigma}^{2}\right)  ,\\
\forall\sigma^{1},\sigma^{2}  &  :U\left(  \sigma^{1},\widehat{\sigma}%
^{2}\right)  \geq\inf_{\sigma^{1}}U\left(  \sigma^{1},\sigma^{2}\right)
\Rightarrow\forall\sigma^{1}:U\left(  \sigma^{1},\widehat{\sigma}^{2}\right)
\geq\sup_{\sigma^{2}}\inf_{\sigma^{1}}U\left(  \sigma^{1},\sigma^{2}\right)
=U\left(  \widehat{\sigma}^{1},\widehat{\sigma}^{2}\right)  .
\end{align*}
After rewriting (to improve clarity) we se that we have proved%
\begin{align*}
\forall\sigma^{2}  &  :U\left(  \widehat{\sigma}^{1},\sigma^{2}\right)  \leq
U\left(  \widehat{\sigma}^{1},\widehat{\sigma}^{2}\right)  ,\\
\forall\sigma^{1}  &  :U\left(  \sigma^{1},\widehat{\sigma}^{2}\right)  \geq
U\left(  \widehat{\sigma}^{1},\widehat{\sigma}^{2}\right)  ,
\end{align*}
which is $\mathbf{C2}$.

To prove that $\mathbf{C2}$ implies $\mathbf{C1}$ we note \cite[p.116,
p.144]{Maschler}\cite[p.43]{Thomas} that we can rewrite $\mathbf{C2}$ as%
\[
\forall\sigma^{1},\sigma^{2}:U\left(  \widehat{\sigma}^{1},\sigma^{2}\right)
\leq U\left(  \widehat{\sigma}^{1},\widehat{\sigma}^{2}\right)  \leq U\left(
\sigma^{1},\widehat{\sigma}^{2}\right)
\]
and then we have
\begin{align*}
&  \sup_{\sigma^{2}}U\left(  \widehat{\sigma}^{1},\sigma^{2}\right)  \leq
U\left(  \widehat{\sigma}^{1},\widehat{\sigma}^{2}\right)  \leq\inf
_{\sigma^{1}}U\left(  \sigma^{1},\widehat{\sigma}^{2}\right)  \Rightarrow\\
&  \inf_{\sigma^{1}}\sup_{\sigma^{2}}U\left(  \sigma^{1},\sigma^{2}\right)
\leq\sup_{\sigma^{2}}U\left(  \widehat{\sigma}^{1},\sigma^{2}\right)  \leq
U\left(  \widehat{\sigma}^{1},\widehat{\sigma}^{2}\right)  \leq\inf
_{\sigma^{1}}U\left(  \sigma^{1},\widehat{\sigma}^{2}\right)  \leq\sup
_{\sigma^{2}}\inf_{\sigma^{1}}U\left(  \sigma^{1},\sigma^{2}\right)
\Rightarrow\\
&  \inf_{\sigma^{1}}\sup_{\sigma^{2}}U\left(  \sigma^{1},\sigma^{2}\right)
=U\left(  \widehat{\sigma}^{1},\widehat{\sigma}^{2}\right)  =\sup_{\sigma^{2}%
}\inf_{\sigma^{1}}U\left(  \sigma^{1},\sigma^{2}\right)
\end{align*}
which is $\mathbf{C1}$.
\end{proof}


\begin{thebibliography}{99}                                                                                               %


\bibitem {Berarducci1993}A. Berarducci and B. Intrigila. \textquotedblleft On
the cop number of a graph.\textquotedblright\ \emph{Advances in Applied
Mathematics}, 14.4 (1993): 389-403.

\bibitem {Berwanger}D. Berwanger, \textquotedblleft Graph games with perfect
information.\textquotedblright\ preprint (2013).

\bibitem {Bonato2011}A. Bonato and R. Nowakowski. \emph{The Game of Cops and
Robbers on Graphs}. American Mathematical Society (2011).

\bibitem {Bonato2017}A. Bonato and G. MacGillivray, \textquotedblleft
Characterizations and algorithms for generalized Cops and Robbers
games\textquotedblright, \emph{Contributions to Discrete Mathematics}, vol.12 (2017).

\bibitem {Fomin2008}F. V. Fomin and D. M. Thilikos, \textquotedblleft An
annotated bibliography on guaranteed graph searching.

\bibitem {Gradel}K.R. Apt and E. Gr\"{a}del. \emph{Lectures in game theory for
computer scientists}. Cambridge University Press, 2011.

\bibitem {Hahn2006}G. Hahn and G. MacGillivray. \textquotedblleft A note on
$k$-cop, $l$-robber games on graphs.\textquotedblright\ \emph{Discrete
Mathematics}, vol. 306.19-20 (2006): 2492-2497.

\bibitem {Hollinger2011}T.H. Chung, , G.A. Hollinger, and V. Isler.
\textquotedblleft Search and pursuit-evasion in mobile
robotics.\textquotedblright\ \emph{Autonomous robots}, vol. 31 (2011): pp.299-310.

\bibitem {filar1996}J. Filar and K. Vrieze. \emph{Competitive Markov decision
processes}. 1996.

\bibitem {CCR}G. Konstantinidis and Ath Kehagias. \textquotedblleft
Simultaneously moving cops and robbers.\textquotedblright\ \emph{Theoretical
Computer Science}, vol. 645 (2016): 48-59.

\bibitem {SCPR}Ath. Kehagias and G. Konstantinidis. \textquotedblleft Selfish
cops and passive robber: Qualitative games.\textquotedblright%
\ \emph{Theoretical Computer Science}, vol. 680 (2017): 25-35.

\bibitem {SCAR}Ath. Kehagias and G. Konstantinidis, \textquotedblleft Selfish
cops and active robber: Multi-player pursuit evasion on
graphs\textquotedblright, \emph{Theoretical Computer Science}, Vol. 780,
pp.84-102, (2019).

\bibitem {GenPurEv}Ath. Kehagias, \textquotedblleft Generalized Cops and
Robbers: A Multi-player Pursuit Game on Graphs.\textquotedblright%
\ \emph{Dynamic Games and Applications}, vol. 9 (2019): 1076-1099.

\bibitem {Konig1927}D. K\"{o}nig, \textquotedblleft Uber eine schlussweise aus
dem endlichen ins unendliche\textquotedblright, \emph{Acta Litt. Ac. Sci.
Hung. Fran. Joseph}, 3 (1927), 121--130.

\bibitem {Maschler}M. Maschler, E. Solan. and S. Zamir. \emph{Game theory}. (2013)

\bibitem {Mazala}R. Mazala. \textquotedblleft Infinite
games.\textquotedblright\ \emph{Automata logics, and infinite games}.
Springer, Berlin, Heidelberg, 2002. pp. 23-38.

\bibitem {Nowakowski1983}R. Nowakowski and P. Winkler. \textquotedblleft
Vertex-to-vertex pursuit in a graph\textquotedblright. \emph{Discrete
Mathematics}, vol. 43 (1983) 235--239.

\bibitem {Quilliot1978}A. Quilliot, Th\`{e}se de 3\`{e}me cycle,
Universit\'{e} de Paris VI, 1978, pp. 131--145.

\bibitem {Quilliot1983}A. Quilliot, Problemes de jeux, de point fixe, de
connectivite\ et de representation sur des graphes, des ensembles ordonnes et
des hypergraphes, 1983.

\bibitem {Quilliot1985}A. Quilliot, \textquotedblleft A short note about
pursuit games played on a graph with a given genus\textquotedblright, \emph{J.
Comb. Theory}, Ser. B, 38 (1985) 89-92.

\bibitem {Schwalbe2001}U. Schwalbe and P. Walker, \textquotedblleft Zermelo
and the early history of game theory\textquotedblright, \emph{Games and
economic behavior}, 34 (2001), no. 1, 123--137.

\bibitem {Thomas}L.C. Thomas. \emph{Games, theory and applications}. Courier
Corporation, 2012.

\bibitem {Ummels}M. Ummels. \textquotedblleft Stochastic multiplayer games:
Theory and algorithms\textquotedblright. Amsterdam University Press, 2010.

\bibitem {Zermelo1913}E. Zermelo, \textquotedblleft\"{U}ber eine Anwendung der
Mengenlehre auf die Theorie des Schachspiels\textquotedblright.
\emph{Proceedings of the fifth international congress of mathematicians}. Vol.
2. II, Cambridge UP, Cambridge, 1913.
\end{thebibliography}
\end{document}